\documentclass{amsart}

\usepackage[all]{xy}
\usepackage{amsmath}
\usepackage{hyperref}
\usepackage{amsfonts,graphics,amsthm,amsfonts,amscd,latexsym}
\usepackage{epsfig}
\usepackage{flafter}
\usepackage{mathtools}
\usepackage{comment}

\hypersetup{
    colorlinks=true,    
    linkcolor=blue,          
    citecolor=blue,      
    filecolor=blue,      
    urlcolor=blue           
}
\usepackage{tikz}
\usetikzlibrary{graphs,positioning,arrows,shapes.misc,decorations.pathmorphing}

\tikzset{
    >=stealth,
    every picture/.style={thick},
    graphs/every graph/.style={empty nodes},
}

\tikzstyle{vertex}=[
    draw,
    circle,
    fill=black,
    inner sep=1pt,
    minimum width=5pt,
]
\usepackage[position=top]{subfig}
\usepackage[alphabetic,backrefs]{amsrefs}
\usepackage{amssymb}
\usepackage{color}

\calclayout

\usetikzlibrary{decorations.pathmorphing}
\tikzstyle{printersafe}=[decoration={snake,amplitude=0pt}]

\newcommand{\codim}{\operatorname{codim}}

\newcommand{\mult}{\operatorname{mult}}

\newcommand{\Spec}{\operatorname{Spec}}

\renewcommand{\qq}{\mathbb{Q}}
\newcommand{\zz}{\mathbb{Z}}
\newcommand{\nn}{\mathbb{N}}
\newcommand{\rr}{\mathbb{R}}

\def\O#1.{\mathcal {O}_{#1}}			%structure sheaf
\def\pr #1.{\mathbb P^{#1}}				%projective space
\def\af #1.{\mathbb A^{#1}}				%affine space
\def\ses#1.#2.#3.{0\to #1\to #2\to #3 \to 0}		%short exact sequence
\def\xrar#1.{\xrightarrow{#1}}			%right arrow with map
\def\K#1.{K_{#1}}						%canonical divisor
\def\bA#1.{\mathbf{A}_{#1}}				%b-divisor A
\def\bM#1.{\mathbf{M}_{#1}}				%b-divisor M
\def\bL#1.{\mathbf{L}_{#1}}				%b-divisor L
\def\bB#1.{\mathbf{B}_{#1}}				%b-divisor B
\def\bK#1.{\mathbf{K}_{#1}}				%b-divisor K
\def\subs#1.{_{#1}}						%subscript
\def\sups#1.{^{#1}}						%supscript

\DeclareMathOperator{\coeff}{coeff}	%coefficients

\DeclareMathOperator{\Supp}{Supp}		%support

\DeclareMathOperator{\glct}{glct}			%generalized lct
\newcommand{\rar}{\rightarrow}		%right arrow
\newcommand{\drar}{\dashrightarrow}	%dashed right arrow

\usepackage{tikz}
\usetikzlibrary{matrix,arrows,decorations.pathmorphing}

  \newtheorem{theorem}{Theorem}[section]
  \newtheorem{lemma}[theorem]{Lemma}
  \newtheorem{proposition}[theorem]{Proposition}
  \newtheorem{corollary}[theorem]{Corollary}
  
  \newtheorem{conjecture}[theorem]{Conjecture}

  \newtheorem{notation}[theorem]{Notation}
  \newtheorem{definition}[theorem]{Definition}
  \newtheorem{example}[theorem]{Example}

\newtheorem{remark}[theorem]{Remark}

\theoremstyle{remark}

\numberwithin{equation}{section}

\begin{document}

\title[Strong $(\delta,n)$-complements]{Strong $(\delta,n)$-complements for semi-stable morphisms}

\author[S. Filipazzi]{Stefano Filipazzi}
\address{
UCLA Mathematics Department,
Box 951555,
Los Angeles, CA 90095-1555, USA
}
\email{filipazzi@math.ucla.edu}

\author[J.~Moraga]{Joaqu\'in Moraga}
\address{Department of Mathematics, Princeton University, Fine Hall, Washington Road, Princeton, NJ 08544-1000, USA
}
\email{jmoraga@princeton.edu}


\begin{abstract}
We prove the boundedness of global strong $(\delta,n)$-complements for generalized $\epsilon$-log canonical pairs of Fano-type.
We also prove some partial results towards boundedness of local strong $(\delta,n)$-complements 
for semi-stable morphisms.
As applications, we prove an effective generalized canonical bundle formula for generalized klt pairs and an effective generalized adjunction formula for exceptional generalized log canonical centers. Moreover, we prove that the existence of strong $(\delta,n)$-complements implies a conjecture due to M$^{\rm c}$Kernan concerning the singularities of the base of a Mori fiber space.
\end{abstract}

\maketitle

\setcounter{tocdepth}{1}
\tableofcontents

\section{Introduction}

The theory of complements was introduced by Shokurov in~\cite{Sho92} to study $3$-fold log flips.
This technique applies to the analysis of contractions of normal varieties $X \rar Z$ with $X$ of Fano-type over a neighborhood of $Z$.
Given a point $z\in Z$ with $X$ $\epsilon$-log canonical over $z$, it is predicted that there exist a positive integer $n$ and a non-negative real number $\delta$, both depending only on $\dim X$ and $\epsilon$, such that the linear system $|-nK_X|$ contains an element $\Gamma$ with $(X,\Gamma/n)$ $\delta$-log canonical over $z$. This conjecture is known as boundedness of strong $(\delta,n)$-complements in dimension $d$.
It is expected that we can take $\delta=\epsilon$~\cite[Conjecture 1.1.3]{Bir04}.
In the above setting, the case $\epsilon=\delta=0$ (resp. $\epsilon\geq \delta>0)$ corresponds to the study of log canonical complements (resp. klt complements).
If $Z$ is the spectrum of an algebraically closed field of characteristic zero, we say that we are in the {\em global setting}; otherwise, we say that we are in the {\em local setting}.

In~\cite{Sho00}, Shokurov introduced more general forms of complements for log pairs
and he proved the existence of bounded $(0,n)$-complements for surfaces.
Then, Prokhorov proved boundedness of $(0,n)$-complements for threefold extremal contractions and threefold conic fibrations~\cites{Pro00,Pro01}.
In~\cite{PS01}, Prokhorov and Shokurov proved that boundedness of $(0,n)$-complements in the local setting in dimension $d$ follows from the minimal model program in dimension $d$ and the existence of bounded global $(0,n)$-complements in dimension $d-1$.
In~\cite{Bir04}, Birkar proved the existence of bounded $(\delta,n)$-complements for surfaces.
Kudryavtsev used the theory of complements to study log del Pezzo surfaces with no discrepancy less than $-\frac{6}{7}$~\cite{Kud04}, and Kudryavtsev and Fedorov classified non-$\qq$-complemented surfaces~\cite{KF04}.
In~\cite{PS09}, Prokhorov and Shokurov proved that boundedness of $(0,n)$-complements follows from the conjecture of effective adjunction and the Borisov--Alexeev--Borisov conjecture (or BAB conjecture for short).

Finally, boundedness of strong $(0,n)$-complements for log canonical pairs was solved by Birkar~\cite[Theorem 1.8]{Bir16a}. This is an essential technique in the proof of the BAB conjecture~\cite[Theorem 1.1]{Bir16b}. In this paper, we conjecture the following version of boundedness of strong $(\delta,n)$-complements for generalized pairs:

\begin{conjecture}\label{encomplement}
Let $d$ and $p$ be two natural numbers, $\epsilon \in [0,1)$, and $\Lambda\subset \qq$ a set satisfying the descending chain condition with rational accumulation points. There exist a natural number $n$
and a non-negative real number $\delta$ only depending on $d,p,\epsilon$ and $\Lambda$ satisfying the following. Let $X\rightarrow Z$ be a contraction between normal quasi-projective varieties, $(X,B+M)$ be a generalized $\epsilon$-log canonical pair of dimension $d$ such that
\begin{itemize}
\item $-(K_X+B+M)$ is nef over $Z$;
\item $X$ is of Fano-type over $Z$;
\item $\coeff(B)\subset \Lambda$; and 
\item $pM'$ is Cartier. 
\end{itemize}
Then, for every point $z\in Z$ there exists a strong $(\delta,n)$-complement for $(X,B+M)$ over $z$.
Moreover, we can pick $\delta>0$ if $\epsilon>0$.
\end{conjecture}

In~\cite[Theorem 1.10]{Bir16a}, Birkar proved Conjecture~\ref{encomplement} in the global setting when $\epsilon=\delta=0$ and $\Lambda$ is a set of hyperstandard coefficients (see, e.g.,~\cite[Section 2.2]{Bir16a}). Moreover, in~\cite[Theorem 1.8]{Bir16a}, Birkar proved Conjecture~\ref{encomplement} in the local setting for $\epsilon=\delta=0$, $M'=0$
and $\Lambda$ a set of hyperstandard coefficients.

We prove the statement of the conjecture in the local setting for generalized log canonical pairs. This is a generalization of~\cite[Theorem 1.8]{Bir16a} to the setting of generalized pairs. 

\begin{theorem}\label{localconj0n}
Conjecture~\ref{encomplement} holds for $\delta=\epsilon=0$.
\end{theorem}

Using techniques introduced by Birkar in~\cite{Bir16a} to prove boundedness of $(0,n)$-complements and boundedness of Fano varieties~\cite[Theorem 1.1]{Bir16b}, we prove the global version of Conjecture~\ref{encomplement}:

\begin{theorem}\label{globalen}
Conjecture~\ref{encomplement} holds if $Z={\rm Spec}(k)$, where $k$ is an algebraically closed field of characteristic zero, in the following cases:
\begin{enumerate}
    \item $\epsilon=0$; 
    \item $M'$ is trivial; or 
    \item $\Lambda$ is finite. 
\end{enumerate}
Moreover, in any case, we can take $\delta=\epsilon$.
\end{theorem}

The third theorem of this paper is a partial result towards boundedness of local strong $(\delta,n)$-complements.
The main techniques involved in the proof of this statement are the theory of semi-stable families and the theory of local log canonical complements.

\begin{theorem}\label{localen}
Let $m$ be a positive integer and $\epsilon$ a positive real number. 
Then, Conjecture~\ref{encomplement} holds if 
$\Lambda$ is finite, $M'$ is trivial, $mK_Z$ is Cartier and $X\rightarrow Z$ is a semi-stable morphism
for the pair $(X,B)$.
Moreover, we can take $\delta > 0$ depending on $m$ and the setup of Conjecture \ref{encomplement}.
\end{theorem}

Now, we turn to discuss some applications of the main theorems. The first application is related to the canonical bundle formula and adjunction formula for generalized pairs. Generalized divisorial adjunction was introduced in~\cites{BZ16,Bir16a}
and then generalized to centers of higher codimension in~\cite{Fil18}. However, in the latter case, it is not known how to control the coefficients of the induced generalized pair. Our first result in this direction is an effective version of the generalized canonical bundle formula:

\begin{theorem}\label{genbundfor}
Let $d$ and $p$ be two natural numbers and $\Lambda \subset \qq$ be a set satisfying the descending chain condition with rational accumulation points. Then, there exist a natural number $q$ and a set $\Omega \subset \qq$ satisfying the descending chain condition with rational accumulation points, only depending on $d,p$ and $\Lambda$, satisfying the following. 
Let $f\colon X\rightarrow Z$ be a contraction between normal quasi-projective varieties,
$(X,B+M)$ be a generalized log canonical pair of dimension $d$ such that 
\begin{itemize}
    \item $K_X+B+M\sim_{\qq,Z} 0$;
    \item $(X,B+M)$ is generalized klt over the generic point of $Z$;
    \item $X$ is of Fano-type over a non-trivial open set $U$ of $Z$;
    \item ${\rm coeff}(B)\subset \Lambda$; and 
    \item $pM'$ is Cartier. 
\end{itemize}
Then there exists a generalized log canonical pair $(Z,B_Z+M_Z)$ on $Z$, so that 
\begin{itemize}
    \item $K_X+B+M\sim_\qq f^*(K_Z+B_Z+M_Z)$, and this formula is preserved under base change;
    \item ${\rm coeff}(B_Z)\subset \Omega$; and 
    \item $qM'_Z$ is Cartier.
\end{itemize}
\end{theorem}

\begin{remark}
{\em The generalized pair $(Z,B_Z+M_Z)$ in Theorem \ref{genbundfor} is induced by $(X,B+M)$ via the generalized canonical bundle formula \cite[Theorem 1.4]{Fil18}.}
\end{remark}

The second result in this direction is an effective version of generalized adjunction to exceptional generalized log canonical centers:

\begin{corollary}\label{genadj}
Let $d$ and $p$ be two natural numbers and $\Lambda \subset \qq$ be a set satisfying the descending chain condition with rational accumulation points. Then, there exist a natural number $q$ and a set $\Omega \subset \qq$ satisfying the descending chain condition with rational accumulation points, only depending on $d,p$ and $\Lambda$, satisfying the following. Let $(X,B+M)$ be a generalized log canonical pair of dimension $d$ such that 
\begin{itemize}
    \item $W\subset X$ is an exceptional generalized log canonical center of $(X,B+M)$;
    \item ${\rm coeff}(B)\subset \Lambda$; and 
    \item $pM'$ is Cartier.
\end{itemize}
Then there exists a generalized pair $(W,B_{W}+M_{W})$ on $W$, so that
\begin{itemize}
    \item $(K_X+B+M)|_{W} \sim_\qq K_{W}+B_{W}+M_{W}$, and this formula is preserved under base change;
    \item ${\rm coeff}(B_{W})\subset \Omega$; and
    \item $qM'_{W}$ is Cartier.
\end{itemize}
\end{corollary}

\begin{remark}
{\em The generalized pair $(W,B_{W}+M_{W})$ in Corollary~\ref{genadj} is induced by $(X,B+M)$ via generalized adjunction \cite[Theorem 1.5]{Fil18}.}
\end{remark}

Finally, we discuss the relation between the existence of klt complements and a conjecture due to M$^{\rm c}$Kernan concerning the singularities of the base of a Mori fiber space~\cite[Conjecture 6.2]{Bir17}.

\begin{conjecture}\label{McKernanconj}
Let $d$ be a natural number and $\epsilon$ a positive real number.
Then there exists a positive real number $\delta$ depending on $d$ and $\epsilon$ such that the following holds. If $X\rightarrow Z$ is a Mori fiber space and $X$ is a $\qq$-factorial projective variety of dimension $d$ with $\epsilon$-log canonical singularities, then $Z$ is $\delta$-log canonical.
\end{conjecture}

Regarding this conjecture, we prove the following statement.

\begin{theorem}\label{McKernanthm}
Conjecture~\ref{encomplement} implies Conjecture~\ref{McKernanconj}. Moreover, Conjecture~\ref{McKernanconj} holds if we allow $\delta$ to depend on the Cartier index of $\K X.$.
\end{theorem}

The main technique involved in the proof of the above proposition is a known case of a conjecture of Shokurov~\cite[Conjecture 6.3]{Bir17}
due to Birkar~\cite{Bir16}.

\subsection*{Acknowledgements} 
The authors would like to thank Christopher Hacon and Karl Schwede for many useful comments.
They would also like to thank an anonymous referee for the careful report and several suggestions.
The first named author was supported by the Graduate Research Fellowship awarded by the University of Utah. The two authors were partially supported by NSF research grants no: DMS-1801851, DMS-1265285 and by a
grant from the
Simons Foundation; Award Number: 256202.
\section{Preliminary results}

Throughout this paper, we work over an algebraically closed field $k$ of characteristic zero.
All varieties considered in this paper are normal unless otherwise stated.
In this section, we will collect some definitions and preliminary results which will be used in this article.

\subsection{Contractions} In this paper a {\em contraction} is a projective morphism of quasi-projective varieties $f  \colon  X \rar Z$ with $f_* \O X. = \O Z.$. Notice that, if $X$ is normal, then so is $Z$.

\subsection{Divisors and linear series} Let $X$ be a normal quasi-projective variety. We say that $D$ is a divisor on $X$ if it is a $\qq$-Weil divisor, i.e., $D$ is a finite sum of prime divisors on $X$ with coefficients in $\qq$. Let $f  \colon  X \rar Z$ be a projective morphism of quasi-projective varieties. Let $D_1$ and $D_2$ be divisors on $X$. We write $D_1 \sim_Z D_2$ (respectively $D_1 \sim \subs \qq,Z . D_2$) if there is a Cartier (respectively $\qq$-Cartier) divisor $L$ on $Z$ such that $D_1 - D_2 \sim f^*L$ (respectively $D_1 - D_2 \sim_\qq f^*L$). Equivalently, we may also write $D_1 \sim D_2$ over $Z$. The case of $\qq$-linear equivalence is also denoted similarly. Let $z$ be a point in $Z$. We write $D_1 \sim D_2$ over $z$ if $D_1 \sim_{Z} D_2$ holds after possibly shrinking $Z$ around $z$. We also make use of the analogous notion for $\qq$-linear equivalence.

\begin{definition} \label{general relative element} {\em 
Let $f \colon  X \rar Z$ be a projective morphism of quasi-projective varieties and $D$ an integral Weil divisor on $X$. Fix a (not necessarily closed) point $z \in Z$. We want to define an appropriate notion of \emph{general element of $|\O X.(D)|$ over $z$}. Without loss of generality, we may take projective closures of $Z$ and $X$, and assume that they are projective varieties. Let $U = \Spec(A) \subset Z$ be an affine open subset containing $z$, and write $X_U \coloneqq f \sups -1. (U)$. Let $H$ be an ample and effective Cartier divisor supported on $Z \setminus U$. In particular, we have $H|_U \sim 0$.

Now, $H^0 (X_U,\O X_U. (D|_U))$ is a finitely generated $A$-module \cite[Theorem II.5.19]{Har77}. Therefore, for any section $s \in H^0 (X_U,\O X_U. (D|\subs X_U.))$ there exists $n \in \nn$ such that $s \in H^0 (X,\O X. (D+nf^*H))$. Thus, we have
$$
H^0 (X_U,\O X_U. (D|_U)) = \bigcup \subs n \in \nn. H^0 (X,\O X. (D+nf^*H)).
$$
Hence, when we refer to a general element of $|\O X.(D)|$ over $z$, we mean a general element of the linear system $|\O X. (D+nf^*H)|$ for $n \gg 0$.}
\end{definition}

\subsection{Generalized pairs and singularities}

In this subsection, we recall the definition of generalized pairs (see, e.g.,~\cite{BZ16}),
which is a generalization of the classic setting of log pairs (see, e.g.,~\cite{KM98}).
We will prove some basic properties regarding the singularities of generalized pairs.

\begin{definition}
{\em A {\em generalized sub-pair} is a triple $(X,B+M)$, where $X$ is a normal quasi-projective variety, 
$K_X+B+M$ is a $\qq$-Cartier divisor, and $M$ is the push-forward of a nef divisor on a higher model of $X$.
More precisely, there exist a projective birational morphism $\pi\colon X'\rightarrow X$ and a nef $\qq$-Cartier divisor $M'$ on $X'$ so that $\pi_*(M')=M$. 
If $B$ is an effective divisor, we will say that $(X,B+M)$ is a {\em generalized pair}. 

In the above situation, we usually call $B$ the {\em boundary part}, $M$ the {\em moduli part} and $B+M$ a {\em generalized boundary}.
We may say that $(X,B+M)$ is a generalized pair with data $X'$ and $M'$.
The divisor $M'$ induces a birational $\qq$-Cartier divisor which descends on $X'$ in the sense of~\cite[Secion 1.7]{Cor07}. 
Hence, we can always replace $X'$ with a higher birational model and $M'$ with its pull-back 
without changing the generalized pair.

Replacing $X'$ with a higher birational model, we may assume that the exceptional locus of $\pi$ is purely of codimension one,
and that the sum of the strict transform of $B$ on $X'$ and the reduced exceptional divisor of $\pi$
is a divisor with simple normal crossing support.
Under these assumptions, we may write
\[
K_{X'}+B'+M'=\pi^*(K_X+B+M), 
\]
where $B'$ is the sum of the strict transform of $B$ and a divisor with support in the exceptional locus of $\pi$.
Given $\epsilon \geq 0$, we will say that the generalized pair $(X,B+M)$ has \emph{generalized $\epsilon$-log canonical singularities} 
if the coefficients of $B'$ are less than or equal to $1-\epsilon$.
If $\epsilon=0$, we omit it from the notation.
}
\end{definition}

\begin{definition}{\em 
Let $(X,B+M)$ be a generalized pair and $\pi \colon Y\rightarrow X$
a projective birational morphism
factoring through $X'$. Then, we may write
\[
K_Y+B_Y+M_Y=\pi^*(K_X+B+M),
\]
where $M_Y$ is the pull-back of $M'$ to $Y$.
Given a prime divisor $E$ on $Y$, we define the {\em generalized log discrepancy of} $(X,B+M)$
with respect to $E$ to be 
\[
a_E(X,B+M) \coloneqq 1-{\rm coeff}_{E}(B_Y).
\]
If the generalized discrepancy of $(X,B+M)$ at $E$ is $\epsilon$, we say that 
$E$ is an {\em $\epsilon$-log canonical place} for the generalized pair, and 
$\pi(E)\subset X$ is an {\em $\epsilon$-log canonical center} for the generalized pair.
If $\epsilon=0$, we omit it from the notation.
If $a_E(X,B+M) \leq 0$, we say that $E$ is a {\em generalized non-klt place} for $(X,B+M)$, and $\pi(E)$ is a {\em generalized non-klt center} for $(X,B+M)$. A generalized log canonical center is said to be \emph{exceptional} if it admits a unique generalized log canonical place and is disjoint from the image of any other non-klt place.}
\end{definition}

\begin{definition}{\em
We say that $(X,B+M)$ is {\em generalized dlt} if $(X,B)$ is dlt and every generalized non-klt center of $(X,B+M)$ is a non-klt center of $(X,B)$. If, in addition, every connected component of $\lfloor B \rfloor$ is irreducible, we say that $(X,B+M)$ is {\em generalized plt}.}
\end{definition}

\begin{definition}
{\em Let $(X,B+M)$ be a generalized pair,
$P$ an effective divisor, and $N$ the push-forward of a nef divisor from a possibly higher birational model of $X$.
Notice that, up to replacing $X'$ with a higher model, we may assume that $M$ and $N$ are the push-forwards of two nef divisors on $X'$.
Assume that the divisor $P+N$ is $\qq$-Cartier.
We define the {\em generalized log canonical threshold} of $K_X+B+M$ with respect to $P+N$ to be
\[
{\rm glct}(K_X+B+M\mid P+N) \coloneqq \sup\{t \mid K_X+B+M+t(P+N)\text{ is generalized log canonical} \},
\]
where $(X,B+M+t(P+N))$ is considered as a generalized pair with boundary part $B+tP$
and moduli part $M+tN$. 
If the above set is empty, then we define the generalized log canonical threshold to be $-\infty$.
Observe that ${\rm glct}(K_X+B+M\mid P+N)$ is non-negative provided that $K_X+B+M$ is generalized log canonical.
Moreover, ${\rm glct}(K_X+B+M\mid P+N)$ is infinite if and only if $N$ descends onto $X$ and $P$ is trivial.
}
\end{definition}

\begin{remark}
{\em 
Given a natural number $p$, 
a set $\Lambda$ of rational numbers satisfying the descending chain condition
and a normal quasi-projective variety $X$,
we denote by $\mathcal{G}\mathcal{B}(X)_{p,\Lambda}$ the set of generalized boundaries $B+M$
on $X$ so that $pM'$ is Cartier, ${\rm coeff}(B)\subset \Lambda$, and $B+M$ is $\qq$-Cartier. 
In~\cite[Theorem 1.5]{BZ16}, Birkar and Zhang proved a version of the ascending chain condition for
generalized log canonical thresholds.
In particular,~\cite[Theorem 1.5]{BZ16} implies that the set 
\[
\mathcal{G}{\rm lct}_{d,p,\Lambda}\coloneqq
\{
{\rm glct}(K_X+B+M\mid P+N) \mid B+M\in \mathcal{G}\mathcal{B}(X)_{p,\Lambda}, P+N \in \mathcal{G}\mathcal{B}(X)_{p,\Lambda},\text{and} \dim X =d
\}
\]
satisfies the ascending chain condition.
This is the ACC for generalized log canonical thresholds that we will use in this article.
}
\end{remark}

\subsection{Bounded families of generalized pairs}

In this subsection, we recall the concept of bounded families of pairs
and introduce the concept of bounded families of generalized pairs.

\begin{definition}\label{def gen lob bdd}{\em
A {\em couple} $(X,D)$ is the datum of a normal projective variety $X$ and a divisor $D$ on $X$ whose coefficients are all equal to one.
A set of couples $\mathcal{Q}$ is said to be {\em log bounded} if there exist 
finitely many projective morphisms $\mathcal{X}^i\rightarrow T^i$ of varieties and reduced divisors $\mathcal{B}^i$ on $\mathcal{X}^i$
so that for every couple $(X,D) \in \mathcal{Q}$ there exist an $i$, a closed point $t\in T^i$
and an isomorphism $\phi \colon \mathcal{X}^i_t \rightarrow X$ so that $(\mathcal{X}^i_t, \mathcal{B}^i_t)$ is a couple
and $\phi^{-1}_*(D) \leq \mathcal{B}^i_t$.
In what follows, we may omit $i$ if it does not play a role in the argument. 
We say that $\mathcal{X} \rightarrow T$ is a {\em bounding family} for $\mathcal{Q}$
and $\mathcal{B}\subset \mathcal{X}$ is a bounding divisor for the set of divisors $\{D \mid (X,D)\in \mathcal{Q}\}$.

A set $\mathcal{P}$ of generalized pairs is said to be {\em generalized log bounded} if there exists 
a log bounded set of couples $\mathcal{Q}$ so that for each $(X,B+M)\in \mathcal{P}$ we can write
$M\sim_\qq \Delta_1 - \Delta_2$, where $\Delta_1$ and $\Delta_2$ are effective $\qq$-divisors,
and $(X,\Supp(B+\Delta_1+\Delta_2))\in \mathcal{Q}$.
If a set of generalized pairs $\mathcal{P}$ is generalized log bounded 
and $M'=0$ for every $(X,B+M)\in \mathcal{P}$, we just say it is {\em log bounded}.
Moreover, if $B=0$ for every $(X,B)\in \mathcal{P}$, we say that $\mathcal{P}$ is {\em bounded}.
}
\end{definition}

\begin{lemma} \label{generalized BAB}
Let $d$ be a natural number and $\epsilon$ be a positive real number. Then the projective varieties $X$ such that
\begin{itemize}
\item $(X,B+M)$ is a generalized $\epsilon$-log canonical pair of dimension $d$ for some $B$ and $M'$; and
\item $-(K_X + B + M)$ is nef and big
\end{itemize}
form a bounded family.
\end{lemma}
\begin{proof}
Let $(X,B+M)$ be as in the statement, and let $\pi  \colon X' \rightarrow X$ be a higher birational model where $M'$ descends. Then, we have that $(X',B')$ is an $\epsilon$-log canonical sub-pair. Fix a rational number $0< \alpha < 1$. Then, the divisor
$$
(1-\alpha)(K_{X'}+B'+M') = K_{X'}+B' + (M'-\alpha(K_{X'}+B'+M'))
$$
is anti-nef and anti-big. Since $-(K_{X'}+ B' + M')$ is nef and big and $M'$ is nef, $M'-\alpha(K_{X'}+B'+M')$ is nef and big. Therefore, there exists an effective $\mathbb{Q}$-divisor $E'$ such that
$$
M'-\alpha(K_{X'}+B'+M') \sim_\mathbb{Q} A'_k + E'_k
$$
for all positive integers $k$, where $A'_k$ is an ample $\mathbb{Q}$-divisor and $E'_k \coloneqq \frac{E'}{k}$ \cite[Example 2.2.19]{Laz04a}. Thus, we may write
$$
(1-\alpha)(K_{X'}+B'+M') \sim_\mathbb{Q} K_{X'}+ B' + A'_k + E'_k.
$$
If we choose $A'_k$ generically in its $\mathbb{Q}$-linear equivalence class and $k \gg 1$, the sub-pair $(X',B'+A'_k+E'_k)$ is $\frac{\epsilon}{2}$-log canonical. Fix such choices. Define $A_k \coloneqq \mu_* A'_k$ and $E_k \coloneqq \mu_*E'_k$. Then, the sub-pair $(X,B+A_k+E_k)$ is an $\frac{\epsilon}{2}$-log canonical pair. Indeed, by construction $B+A_k+E_k$ is effective, and the sub-pair $(X',B'+A'_k+E'_k)$ is the crepant pull-back of $(X,B+A_k+E_k)$. Then, we conclude that
$$
(1-\alpha) (K_X + B  + M) \sim_\mathbb{Q} K_X + B + A_k + E_k.
$$
Thus, we have that the pair $(X,B+A_k+E_k)$ is weak log Fano and $\epsilon/2$-log canonical. 
By \cite[Theorem 1.1]{Bir16b}, the $X$ as in the statement belongs to a bounded family.
\end{proof}

The following is a consequence of the proof of \cite[Proposition 2.4]{HX15}. We include a proof for the reader's convenience.

\begin{proposition} \label{prop Q-factorial family}
Let $\lbrace (X_i,B_i) \rbrace \subs i \geq 1.$ be a sequence of $\epsilon$-log canonical $\qq$-factorial pairs, where $\epsilon > 0$. Assume that there exist a projective morphism $\pi  \colon  \mathcal{X} \rar T$ to a variety of finite type, a divisor $\mathcal{B}$ on $\mathcal{X}$ and a dense sequence of closed points $t_i$ on $T$ so that
$(X_i,B_i) \cong (\mathcal{X}_i,\mathcal{B}_i)$ as pairs, where $\mathcal{X}_i \coloneqq \mathcal{X} \subs t_i.$, and $\mathcal{B}_i \coloneqq \mathcal{B}| \subs \mathcal{X}_{t_i}.$.
Then, there exist a birational morphism $f  \colon  \mathcal{X}' \rar \mathcal{X}$, a divisor $\mathcal{B}'$ on $\mathcal{X}'$ and a dense open set $U \subset T$ such that:
\begin{itemize}
\item $(\mathcal{X}',\mathcal{B}')$ is a $\qq$-factorial klt pair, with $f_* \mathcal{B}'= \mathcal{B}$; and
\item $f$ is small over $U$.
\end{itemize}
In particular, we have $(\mathcal{X}' \subs i_k.,\mathcal{B}' \subs i_k.) \cong (X \subs i_k., B \subs i_k.)$ for a dense subsequence $t_{i_k}$ in $T$.
\end{proposition}

\begin{proof}
We follow the proof of \cite[Proposition 2.4]{HX15}. Fix a rational number $0 < \epsilon' < \epsilon$. Up to shrinking $T$, there is a log resolution $g  \colon  \mathcal{Y} \rar \mathcal{X}$ of $(\mathcal{X},\mathcal{B})$ such that $(\mathcal{Y},g_* \sups -1. \mathcal{B} + \mathcal{E})$ is log smooth over $T$, where $\mathcal{E}$ denotes the reduced exceptional divisor for $g$. We run a $(\K \mathcal{Y}. + g_* \sups -1. \mathcal{B} + (1-\epsilon') \mathcal{E})$-minimal model program over $\mathcal{X}$. By \cite{BCHM}, it terminates with a minimal model $f  \colon  \mathcal{X}' \rar \mathcal{X}$. Let $h  \colon  \mathcal{Y} \drar \mathcal{X}'$ denote the induced birational map.

Observe that the divisor
\[
\K \mathcal{X}'. + h_*(g_* \sups -1. \mathcal{B} + (1-\epsilon') \mathcal{E}))
\]
is nef over $\mathcal{X}$, and so is 
\[
(\K \mathcal{X}'. + h_*(g_* \sups -1. \mathcal{B} + (1-\epsilon') \mathcal{E})))| \subs \mathcal{X}'_i.
\]
over $\mathcal{X}_i$. Let $f_i  \colon  \mathcal{X}'_i \rar \mathcal{X}_i$ denote the corresponding birational morphism. Since 
\[
(\K \mathcal{X}'. + h_*(g_* \sups -1. \mathcal{B} + (1-\epsilon') \mathcal{E}))| \subs \mathcal{X}'_i. - f_i^* (\K \mathcal{X}_i. + \mathcal{B}_i)
\]
is effective, $f_i$-nef and supported on all the $f_i$-exceptional divisors, by the negativity lemma~\cite[Lemma 3.39]{KM98} it follows that $f_i$ is a small birational morphism \cite[Lemma 3.39]{KM98}. Since $\mathcal{X}_i$ is $\qq$-factorial, then $f_i$ is the identity morphism \cite[Corollary 2.63]{KM98}.

Since all the restrictions $f_i$ are small, no $f$-exceptional divisor dominates $T$. Therefore, up to shrinking $T$, we may assume that $f$ is small. Thus, we have that 
\[
f_* \sups -1. \mathcal{B}= h_*(g_* \sups -1. \mathcal{B} + (1-\epsilon') \mathcal{E})).
\]
Set $\mathcal{B}' \coloneqq f_* \sups -1. \mathcal{B}$. Then, by construction $(\mathcal{X}',\mathcal{B}')$ is a $\qq$-factorial klt pair.
\end{proof}

\subsection{Theory of complements}

In this subsection, we give the basic definitions related to the theory of complements.
We will start by defining varieties of relative Fano-type, which is the class of varieties of interest for this work.

\begin{definition}\label{FT}{\em 
Let $X\rightarrow Z$ be a projective morphism between quasi-projective varieties.
We say that $X$ is of {\em Fano-type} over $Z$ if there exists a boundary $B$ on $X$ such that
$(X,B)$ is a klt pair and $-(K_X+B)$ is nef and big over $Z$.
It is known that if $X$ is of Fano-type over $Z$, then any minimal model program over $Z$ for a divisor $D$ on $X$
terminates~\cite[Corollary 1.3.2]{BCHM}.
}
\end{definition}

\begin{definition}\label{comp}{\em
Let $(X,B+M)$ be an $\epsilon$-log canonical generalized pair and $X\rightarrow Z$ a contraction of normal quasi-projective varieties. 
We say that the divisor $B^+$ is a {\em $(\delta,n)$-complement} over $z\in Z$ if the following conditions hold over some neighborhood of $z$:
\begin{itemize}
\item $(X,B^{+}+M)$ is a $\delta$-log canonical generalized pair with boundary part $B^{+}$;
\item $n(K_X+B^{+}+M)\sim 0$; and 
\item $nB^{+}\geq n\lfloor B \rfloor + \lfloor (n+1)\left\{ B \right\} \rfloor$.
\end{itemize}
If $nB^+\geq nB$, then we say that $B^+$ is a {\em strong $(\delta,n)$-complement}.
}
\end{definition}

\begin{remark} \label{rmk_klt_complement}
{\em
Let $(X,B+M)$ be a generalized klt pair and $X\rightarrow Z$ be a contraction of normal quasi-projective varieties. 
Let $B^+$ be a $(0,n)$-complement over $z\in Z$.
Then, if $(X,B^+ +M)$ is generalized klt over a neighborhood of $z \in Z$, $B^+$ is a $\left( \frac{1}{n},n \right)$-complement over $z \in Z$.
Indeed, by definition, $n(\K X. + B^+ + M)$ is Cartier over a neighborhood of $z \in Z$.
Thus, the generalized log discrepancies of $(X,B^+ + M)$ are integer multiples of $\frac{1}{n}$.
Since $(X,B^+ + M)$ is klt, its generalized log discrepancies are at least $\frac{1}{n}$.
}
\end{remark}

\begin{lemma} \label{lemma strong complements finite coeff}
Let $d$ and $p$ be natural numbers, $\epsilon \in [0,1)$, and $\Lambda \subset [0,1]$ a finite set of rational numbers. Then there is a natural number $n$ depending only on $d$, $p$, $\epsilon$ and $\Lambda$ such that the following holds. Assume that $(X,B+M)$ is a projective generalized $\epsilon$-log canonical pair of dimension $d$ such that:
\begin{itemize}
\item $X$ is of Fano-type;
\item $\coeff(B) \subset \Lambda$;
\item $pM$ is an integral divisor; and
\item $-(\K X. + B + M)$ is nef.
\end{itemize}
Then there is a global strong $(\epsilon,n)$-complement for $(X,B+M)$.
\end{lemma}

\begin{proof}
By \cite[Theorem 1.10]{Bir16a}, we may assume that $\epsilon > 0$.
Under this assumption, we claim that the varieties $X$ as in the statement form a bounded family.
Indeed, since $X$ is of Fano-type we can find a boundary divisor $\Gamma$ on $X$ so that 
$(X,\Gamma)$ is klt and $-(K_X+\Gamma)$ is nef and big. Therefore, we conclude that the generalized pair
\[
K_X+\frac{B+\Gamma}{2}+\frac{M}{2}
\]
is $\frac{\epsilon}{2}$-generalized log canonical 
and 
\[
-\left(K_X+\frac{B+\Gamma}{2}+\frac{M}{2} \right)
\]
is nef and big. Hence, by Lemma~\ref{generalized BAB} we know that the varieties $X$ belong to a bounded family.

Therefore, for any such $X$, there is a very ample Cartier divisor $A$ such that $A^d \leq r$ and $A \sups d-1. \cdot(-\K X.) \leq r$ for some fixed number $r$. Since $B \geq 0$ and $M$ is the push-forward of a nef divisor, it follows that $A \sups d-1. \cdot (-\K X. - B - M) \leq r$. Since $\K X. + B + M$ has bounded Weil index $c$ depending just on $\Lambda$ and $p$, up to replacing $r$ by $cr$, we may also assume that $A \sups d-1. \cdot (-c\K X. - cB - cM) \leq r$.

By \cite[Lemma 2.25]{Bir16a}, $\K X. + B + M$ has bounded Cartier index, which we will denote by $a$.
Thus, we can apply the effective basepoint-free theorem~\cite[Theorem 1.1]{Kol93} for the Cartier divisor $a(K_X+B+M)$
on the klt pair $(X,\Gamma)$, since the $\qq$-Cartier divisor 
\[
-a(K_X+B+M) - (K_X+\Gamma)
\]
is nef and big.
Thus, there is a uniform positive integer $n$, divisible by $a$, such that $|-n(\K X. + B + M)|$ is basepoint-free. Without loss of generality, we may assume that $n \geq (1-\epsilon) \sups -1.$.

Now, let $D$ be a general element of $|-n(\K X. + B + M)|$. Define $B^+ \coloneqq B + \frac{D}{n}$. Since $|-n(\K X. + B + M)|$ is free, $(X,B+M)$ is generalized $\epsilon$-lc and $n \geq (1-\epsilon) \sups -1.$, it follows that $(X,B^++M)$ is generalized $\epsilon$-lc. By construction, $\K X. + B^+ + M$ is a global $(\epsilon,n)$-complement for $\K X. + B + M$. As $B^+ \geq B$, it is automatically a strong complement. 
\end{proof}

\begin{proposition} \label{replace complement in linear series}
Let $\Lambda \subset [0,1]$ be a finite set of rational numbers, $p$ and $n$ a positive integers and $\epsilon \geq 0$ a rational number. Let $\mathcal{P}$ be a set of generalized pairs $(X,B+M)$ such that
\begin{itemize}
\item $\coeff(B) \subset \Lambda$, and $pM'$ is Cartier;
\item there is a contraction $X \rar Z$, with $X$ of Fano-type over $Z$; and
\item $(X,B+M)$ admits a strong $(\epsilon,n)$-complement over any point $z \in Z$.
\end{itemize}
Then, for any sufficiently divisible positive integer $m$, depending on $\mathcal{P}$, in the above setup any $(X,B+M) \in \mathcal{P}$ admits a strong $(\epsilon,mn)$-complement $B^+$ with 
\[
mn(B^+-B) \in |-mn(\K X. + B + M)|
\]
a general element of the linear system.
\end{proposition}
\begin{proof}
Let $m$ be a positive integer such that $m \Lambda \subset \nn$ and $p|m$. Fix $(X,B+M) \in \mathcal{P}$, and let $B^+$ be a strong $(\epsilon,n)$-complement as in the statement. Since $B^+ \geq B$, it is automatically a strong $(\epsilon,mn)$-complement. We may assume that $Z$ is affine.
Therefore, we have $mn(B^+-B) \in |-mn(\K X. + B + M)|$. In particular, $|-mn(\K X. + B + M)| \neq \emptyset$. Fix $E \in |-mn(\K X. + B + M)|$. Then, we have
\begin{equation} \label{equation linear equivalence}
mn(\K X. + B^+ + M) \sim mn\left( \K X. + B + \frac{E}{mn} + M\right).
\end{equation}
In particular, the right hand side of \eqref{equation linear equivalence} is Cartier.
By Definition \ref{general relative element}, we can regard the element $mn(B^+-B)$ of $|-mn(\K X. + B + M)|$ over $z$ as an element of a linear series on a projective closure of $X$. Thus, by \cite[Corollary 2.33]{KM98}, for a general choice of $E$ in the same linear series, the singularities of $(X,B+\frac{E}{mn}+M)$ are not worse than the ones of $(X,B^+ + M)$. Hence, the claim follows.
\end{proof}

\subsection{Examples of complements}

In this subsection, we give some examples of complements.
In particular, we show that the conditions of Conjecture~\ref{encomplement} are necessary for the existence of strong complements. 

\begin{example}{ \em
Let $\lbrace \alpha_i \rbrace \subs i \geq 1.$ be a strictly increasing sequence of rational numbers with $\lim \subs i \to +\infty. \alpha_i = \frac{\sqrt{2}}{2}$. Similarly, let $\lbrace \beta_i \rbrace \subs i \geq 1.$ be a strictly increasing sequence of rational numbers with $\lim \subs i \to +\infty. \beta_i =1 - \frac {\sqrt{2}}{2}$. Define $\Lambda \coloneqq \lbrace \alpha_i \rbrace \subs i \geq 1. \cup \lbrace \beta_i \rbrace \subs i \geq 1.$. Then, $\Lambda$ is a set of rational numbers satisfying the descending chain condition. Notice that the accumulation points are not rational.

Fix four distinct closed points $P,Q,R,S \in \pr 1.$. Consider the sequence of boundaries $\Delta_i \coloneqq \alpha_i P + \alpha_i Q + \beta_i R + \beta_i S$. Then, $(\pr 1.,\Delta_i)$ is klt and $-(\K {\pr 1.}.+\Delta_i)$ is ample. We will show that $(\pr 1.,\Delta_i)$ does not admit a bounded strong log canonical complement.

Fix $n \in \nn$. Then, for $i \gg 1$ we have $\frac{\lceil n \alpha_i \rceil}{n} > \frac{\sqrt{2}}{2}$. Similarly, we have $\frac{\lceil n \beta_i \rceil}{n} > 1- \frac{\sqrt{2}}{2}$. Therefore, there exists no $\Gamma \geq \Delta_i$ such that $n\Gamma$ in integral and $\deg \Gamma =2$. In particular, there exists no strong $(0,n)$-complement for $(\pr 1., \Delta_i)$.
}
\end{example}

\begin{example}{\em
Define $\Lambda \coloneqq \lbrace 1-\frac{1}{n}| n \in \nn,n \geq 2 \rbrace \cup \lbrace \frac{\sqrt{2}}{2},1-\frac{\sqrt{2}}{2}  \rbrace$. Notice that $\Lambda$ is a set satisfying the descending chain condition, and that all of the accumulation points are rational.

Fix three distinct closed points $P,Q,R \in \pr 1.$, and define boundaries $\Delta_i \coloneqq \frac{i-1}{i}P+\frac{\sqrt{2}}{2}Q+ (1-\frac{\sqrt{2}}{2})R$. Then, $(\pr 1.,\Delta_i)$ is klt and $-(\K {\pr 1.}.+\Delta_i)$ is ample. We will show that $(\pr 1.,\Delta_i)$ does not admit a bounded strong log canonical complements.

Fix $n \in \nn$. Then, for $i > n$ we have $\frac{\lceil n \frac{i-1}{i} \rceil}{n} = 1$. Therefore, any divisor $\Gamma \geq \Delta_i$ with $n\Gamma$ integral satisfies $\Gamma > P+\frac{\sqrt{2}}{2}Q+ (1-\frac{\sqrt{2}}{2})R$. In particular, we have $\deg \Gamma > 2$.
Hence, there exists no strong $(0,n)$-complement for $(\pr 1., \Delta_i)$.
}
\end{example}

\begin{example}{\em
Define $\Lambda \coloneqq \lbrace \frac{1}{n}|n\in \nn, n \geq 1 \rbrace$. Fix a sequence of distinct closed points $\lbrace P_j \rbrace \subs j \geq 1. \subset \pr 1.$. Define boundaries $\Delta_i \coloneqq \frac{1}{i} \sum \subs j=1. \sups i-1. P_j$. Then, $(\pr 1.,\Delta_i)$ is klt, and $-(\K {\pr 1.}. + \Delta_i)$ is ample.

Fix $n \in \nn$. Then, for $i \gg 1$ we have $\deg \frac{\lceil n \Delta_i \rceil}{n}>2$. Hence, there exists no strong $(0,n)$-complement for $(\pr 1., \Delta_i)$.
}
\end{example}

\begin{example}{\em
Define $\Lambda \coloneqq \lbrace 1-\frac{\sqrt{2}}{2},\frac{\sqrt{2}}{2}\rbrace$. Fix four distinct closed points $P,Q,R,S \in \pr 1.$. Then, set $\Delta \coloneqq (1-\frac{\sqrt{2}}{2})(P+R) + \frac{\sqrt{2}}{2}(R+S)$. The pair $(\pr 1., \Delta)$ is log canonical, $-(\K {\pr 1.}. + \Delta)$ is nef and $\pr 1.$ is Fano. Since $\deg \Delta =2$, $(X,\Delta)$ does not admit a strong $(0,n)$-complement for any $n \in \nn$.
}
\end{example}

The above examples show that to develop a theory of bounded strong complements, we need to fix a set of coefficients $\Lambda$ satisfying the descending chain condition. Furthermore, if $\Lambda$ is infinite, we need $\overline{\Lambda}\subset \qq$.

\subsection{Semi-stable families}

In this subsection, we recall some properties of semi-stable morphisms for pairs. We refer to \cite[Definition-Lemma 5.10]{Kol13} for the definition of semi-log canonical pair. Let $(X,B)$ be a pair, and let $f  \colon  X \rar Z$ be a flat, projective and surjective morphism of quasi-projective varieties. We say that {\em $f  \colon  (X,B) \rar Z$ is a semi-stable family of semi-log canonical pairs} if the following conditions are satisfied:
\begin{itemize}
    \item $\Supp(B)$ avoids the generic and codimension one singular points of every fiber;
    \item $\K X/Z. + B$ is $\qq$-Cartier, where $\K X/Z.$ denotes the relative canonical divisor; and
    \item $X_z$ is reduced and $(X_z,B_z)$ is a connected semi-log canonical pair for all $z \in Z$.
\end{itemize}
Equivalently, we say that $f \colon  X \rar Z$ is a {\em semi-stable morphism for the pair $(X,B)$}. In case $B=0$, we just say that $f  \colon  X \rar Z$ is a {\em semi-stable morphism}.

\begin{proposition} \label{base change normal}
Let $f\colon X \rar Z$ be a semi-stable morphism of normal varieties. Let $\pi\colon Z' \rar Z$ be a birational contraction with $Z'$ normal. Write $X' \coloneqq  X \times_Z Z'$. Then, $X'$ is a normal variety.
\end{proposition}

\begin{proof}
Write $g  \colon  X' \rar Z'$ and $\psi  \colon  X' \rar X$ for the induced morphisms. Let $\mathcal{I} \subset \O Z.$ be the ideal sheaf corresponding to $\pi$ \cite[Theorem II.7.17]{Har77}. Then, since $f$ is flat, $\psi$ is induced by the blow-up of the ideal sheaf $f \sups -1. \mathcal{I} \subset \O X.$. In particular, $X'$ is an irreducible variety \cite[Proposition II.7.16]{Har77}.

Since $f$ is semi-stable, it is an $S_2$ morphism \cite[D\'efinition 6.8.1]{EGAIV2}. Then, by base change, we have that $g$ is an $S_2$ morphism \cite[Proposition 6.8.2]{EGAIV2}. Therefore, since $g$ is $S_2$ and $Z'$ is normal, we conclude that $X'$ is $S_2$ \cite[Proposition 6.8.3]{EGAIV2}.

We are left with showing that $X'$ is $R_1$. Notice that the general fiber of $g$ is normal. Therefore, there is a closed subset $W \coloneqq \lbrace x \in X|X \subs f(x). \; \mathrm{not} \; \mathrm{normal} \; \mathrm{at} \; x \rbrace \subset X$ that contains no fiber and that does not dominate $Z$. Write $W' \coloneqq \psi \sups -1. (W)$. Then, $W'$ has the same property, and therefore $\codim \subs X'. W' \geq 2$. Then, by \cite[Proposition 6.8.3]{EGAIV2}, $X' \setminus W'$ is normal. Since $X'$ is $S_2$, we conclude that it is normal.
\end{proof}

\subsection{Generalized canonical bundle formula}\label{subsec:canbund}

In this subsection, we recall the construction of the generalized canonical bundle formula introduced in \cite{Fil18}. Let $(X,B+M)$ be a generalized sub-pair, and let $f  \colon  X \rar Z$ be a contraction where $\dim Z > 0$. Assume that $(X,B+M)$ is sub-log canonical over the generic point of $Z$ and that $\K X. + B + M \sim \subs \qq,f. 0$. Fix a divisor $L_Z$ on $Z$ such that $\K X. + B + M \sim_\qq f^*L_Z$. Then, for any prime divisor $D$ on $Z$, let $t_D$ be the generalized log canonical threshold of $f^*D$ with respect to $(X,B+M)$ over the generic point of $D$. Then, set $B_Z \coloneqq \sum b_D D$, where $b_D \coloneqq 1-t_D$. Define $M_Z \coloneqq L_Z - (\K Z. + B_Z)$. Hence, we can write
$$
\K X. + B + M \sim_\qq f^*(\K Z. + B_Z + M_Z).
$$

Now, let $\tilde X$ and $\tilde Z$ be higher birational models of $X$ and $Z$ respectively, and assume we have a commutative diagram of morphisms as follows

\begin{center}
\begin{tikzpicture}
\matrix(m)[matrix of math nodes,
row sep=2.6em, column sep=2.8em,
text height=1.5ex, text depth=0.25ex]
{\tilde X & X\\
\tilde Z & Z\\};
\path[->,font=\scriptsize,>=angle 90]
(m-1-1) edge node[auto] {$\phi$} (m-1-2)
edge node[auto] {$g$} (m-2-1)
(m-1-2) edge node[auto] {$f$} (m-2-2)
(m-2-1) edge node[auto] {$\psi$} (m-2-2);
\end{tikzpicture}
\end{center}

We denote by $(\tilde X, \tilde B  + \tilde M)$ the trace of the generalized sub-pair $(X,B+M)$ on $\tilde X$. Furthermore, set $L_{\tilde Z} \coloneqq \psi^* L_{Z}$. With this piece of data, we can define divisors $B_{\tilde Z}$ and $M_{\tilde Z}$ such that
$$
\K \tilde X. + \tilde B + \tilde M \sim_\qq g^*(\K \tilde Z. + B_{\tilde Z} + M_{\tilde Z}),
$$
$B_{Z}= \psi_* B_{\tilde Z}$ and $M_{Z}= \psi_* M_{\tilde Z}$. In this way, b-divisors $\mathbf{B} \subs Z.$ and $\mathbf{M} \subs Z.$ are defined. We write $\bB Z,\tilde{Z}.$ and $\bM Z,\tilde{Z}.$ for the traces of $\bB Z.$ and $\bM Z.$ on any higher model $\tilde{Z}$.

In this setup, we have the following theorem, referred to as the generalized canonical bundle formula.

\begin{theorem}[cf.{\cite[Theorem 1.4]{Fil18}}]\label{generalized canonical bundle formula}
Let $(X',B'+M')$ be a generalized sub-pair with data $X \rar X'$ and $M$. Assume that $B'$, $M'$ and $M$ are $\qq$-divisors. Let $f  \colon  X' \rar Z'$ be a contraction such that $\K X'. + B' + M' \sim \subs \qq, f. 0$. Also, let $(X',B'+M')$ be generalized log canonical over the generic point of $Z'$. Then, the b-divisor $\bM Z'.$ is $\qq$-Cartier and b-nef.
\end{theorem}

\subsection{Generalized adjunction}\label{subsec:genadj}

In this subsection, we recall how to define adjunction for generalized pairs. Let $(X,B+M)$ be a generalized pair, and let $S \subset X$ be a prime divisor in the support of $\lfloor B \rfloor$. Denote by $S^\nu$ the normalization of $S$. Then, consider a log resolution $\pi \colon  X' \rar X$ of $(X,B)$ where $M'$ descends. Set $S' \coloneqq \pi_* \sups -1. S$. By adjunction, we can define the sub-pair $(S',B \subs S'.)$. Then, we define $M \subs S'. \coloneqq M'| \subs S'.$. Therefore, we can regard $(S',B \subs S'. + M \subs S'.)$ as a generalized sub-pair. Then, let $\rho  \colon  S' \rar S^\nu$ be the induced morphism, and set $B \subs S^\nu. \coloneqq \rho_* B \subs S'.$ and $M \subs S^\nu. \coloneqq \rho_* M \subs S'.$. In this way, $(S^\nu,B \subs S^\nu. + M \subs S^\nu.)$ becomes a generalized pair. We refer to this operation as divisorial generalized adjunction.

More generally, let $W \subset X$ be an exceptional generalized log canonical center, and denote by $W^\nu$ its normalization.
To define a generalized pair on $W^\nu$, we argue as follows.
Fix the generalized log canonical place $E'$ dominating $W$, and let $\pi \colon  X' \rar X$ be a higher model where $E'$ appears as a normal prime divisor.
By generalized divisorial adjunction, $E'$ inherits a generalized sub-pair structure $(E',B \subs E'. + M \subs E'.)$ from $(X',B'+M')$.
Then, we consider the induced fibration $\rho  \colon  E' \rar W^\nu$.
Finally, we apply Theorem \ref{generalized canonical bundle formula} to induce a generalized pair structure on $W^\nu$.
From the construction, it follows that $W^\nu=W'$.

In this setup, we have the following statement, referred to as generalized adjunction and inversion thereof.

\begin{theorem}[cf.{\cite[Theorem 1.6]{Fil18}}] \label{generalized adj and inv of adj}
Let $(X',B'+M')$ be a generalized pair with data $X \rar X'$ and $M$. Assume that $B'$, $M$ and $M'$ are $\qq$-divisors. Let $W'$ be a generalized log canonical center of $(X',B'+M')$ with normalization $W^\nu$.
Assume that a structure of generalized pair $(W^\nu,\bB W^\nu. + \bM W^\nu.)$ is incuded on the normalization $W^\nu$ of $W'$.
Then, $(W^\nu,\bB W^\nu. + \bM W^\nu.)$ is generalized log canonical if and only if $(X',B'+M')$ is generalized log canonical in a neighborhood of $W'$.
\end{theorem}

\section{Global strong $(\delta,n)$-complements}

In this section, we prove Theorem~\ref{globalen}.
We start by proving Lemma~\ref{incrcoeff}, which will allow us to perturb the coefficients of a generalized pair while keeping it generalized log canonical.

\begin{notation} \label{notation}
{\em 
Let $\Lambda \subset (0,1]$ be a set of rational numbers satisfying the descending chain condition. 
Given a natural number $m\in \mathbb{N}$,
we consider the partition 
\[
\mathcal{P}_m \coloneqq \left\{ \left(0, \frac{1}{m}\right], \left( \frac{1}{m}, \frac{2}{m} \right], \dots, \left( \frac{m-1}{m}, 1 \right]\right\}
\]
of the interval $(0,1]$.
For each $b\in \Lambda$, we denote by $I(b,m)$ the interval in the partition $\mathcal{P}_m$ such that $b \in I(b,m)$.
Then, for each $b \in \Lambda$, we define
\[
b_m \coloneqq \sup\left\{ x \mid x\in I(b,m) \cap \Lambda \right\}.
\]
Observe that, for any $b\in \Lambda$ and $m$ a positive integer, we have the inequality $b\leq b_m$, since $b \in I(b,m) \cap \Lambda$.
Furthermore, if $m$ is divisible enough, we have $b = \max \lbrace x \mid x \in I(b,m) \rbrace$, which implies the equality $b=b_m$.
We denote by $\mathcal{C}_m=\{b_m \mid b\in \Lambda\}$. Observe that for each $m$ the set $\mathcal{C}_m$ is finite, and the set
\[
\overline{\Lambda} = \bigcup_{m \in \mathbb{N}} \mathcal{C}_m
\]
satisfies the descending chain condition.
Given a boundary divisor $B\geq 0$ on a normal quasi-projective variety $X$, we can write $B=\sum_j b^{(j)} B^{(j)}$ in a unique way such that the $B^{(j)}$ are pairwise different prime divisors on $X$.
Whenever the coefficients of the divisor $B$ belong to $\Lambda$, we can define the boundary divisor $B_m \coloneqq \sum_j b_{m}^{(j)} B^{(j)}$.
By the above discussion, it follows that $B \leq B_m$.
}
\end{notation}

The following lemma is a generalization of~\cite[Proposition 2.50]{Bir16a} 
for $\Lambda$ having rational accumulation points.

\begin{lemma}\label{incrcoeff}
Let $d$ and $p$ be two natural numbers, and $\Lambda \subset \qq$ be a set satisfying the descending chain condition with rational accumulation points. There exists a natural number $m$, only depending on $d,p$, and $\Lambda$ satisfying the following. 
If $X\rightarrow Z$ is a contraction between normal quasi-projective varieties, $(X,B+M)$ is a generalized log canonical pair of dimension $d$ such that 
\begin{itemize}
\item there exists a divisor $\Omega \geq 0$ so that $(X,B+\Omega+M)$ is generalized log canonical and we have $\K X. + B + \Omega + M \sim \subs \qq,Z. 0$;
\item $X$ is $\qq$-factorial of Fano-type over $Z$;
\item ${\rm coeff}(B)\subset \Lambda$; and 
\item $pM'$ is Cartier.
\end{itemize}
Let $B_m$ be as in Notation \ref{notation}.
Then, the following conditions hold
\begin{itemize}
    \item $(X,B_m+M)$ is generalized log canonical; 
    \item we may run a minimal model program for  $-(K_X+B_m+M)$ over $Z$ that terminates with a generalized log canonical pair $(X'',B''_m+M'')$; and 
    \item $-(K_{X''}+B''_m+M'')$ is nef over $Z$.
\end{itemize}
Moreover, if $-(K_{X''}+B''_m+M'')$ has a strong $(\epsilon,n)$-complement over $z\in Z$, then so does $-(K_X+B+M)$.
\end{lemma}

\begin{proof}
We will prove each statement independently by contradiction applying the ascending chain condition for generalized log canonical thresholds~\cite[Theorem 1.5]{BZ16} and the global ascending chain condition for generalized log canonical pairs~\cite[Theorem 1.6]{BZ16}.
\begin{enumerate}
\item Assume it is not true. Then, there exists a sequence of generalized log canonical pairs $(X_i,B_i+M_i)$ as in the statement such that $(X_i,B_{i,i}+M_i)$ is not generalized log canonical, where $B \subs i,i.$ is obtained from $B_i$ as in Notation \ref{notation}.
We claim that we can find boundaries 
$B_i \leq \Gamma_{i} \leq B_{i,i}$ and prime divisors $D_i$ such that 
\begin{equation}\label{ineqcoeff}
{\rm coeff}_{D_i}(B_i) \leq {\rm coeff}_{D_i}(\Gamma_{i}) < {\rm coeff}_{D_i}(B_{i,i}),
\end{equation}
all the remaining coefficients of $\Gamma_i$ belong to $\overline{\Lambda}$,
and 
\[
{\rm coeff}_{D_i}(\Gamma_i) = {\rm glct}(K_{X_i}+B_i+M_i \mid D_i).
\]
In what follows we will write 
\[
B_i = \sum_j b_i^{(j)} B_{i}^{(j)},
\]
where the $B_{i}^{(j)}$ are pairwise different
prime divisors and $b^{(j)}_i \in \Lambda$. 
We will produce $\Gamma_i$ by successively increasing the coefficients of $B_i$
which differ from the coefficients of $B_{i,i}$. 
Indeed, if 
\[
{\rm glct}(K_{X_i}+B_i+M_i \mid B_i^{(1)}) \geq b_{i,i}^{(1)}-b_i^{(1)},
\]
then we can increase the coefficient $b_i^{(1)}$ of $B_i^{(1)}$ to  $b_{i,i}^{(1)}$ 
and the generalized pair will remain generalized log canonical.
By abusing notation, we will denote the new boundary by $B_i$.
Then, we proceed inductively with the other coefficients.
Since $(X_i,B_{i,i}+M_i)$ is not generalized log canonical,
we will eventually find $j_i$ so that 
\[
\beta_i^{(j_i)}={\rm glct}(K_{X_i}+B_i+M_i \mid B_i^{(j_i)}) < b_{i,i}^{(j_i)}-b_{i}^{(j_i)},
\]
so we may increase $b_i^{(j_i)}$ to $\beta_i^{(j_i)}$ to obtain the desired $\Gamma_i$ with $D_i= B_i^{(j_i)}$.
We denote by $\Gamma'_i$ the divisor obtained from $\Gamma_i$ by reducing the coefficient of $D_i$ to zero.
Observe that the coefficients of $\Gamma'_i$ belong to the set $\overline{\Lambda}$, which satisfies the descending chain condition.

Now, we claim that the generalized log canonical thresholds of $(X_i, \Gamma'_i +M_i)$ with respect to $D_i$
form an infinite increasing sequence. This will provide the required contradiction.
Let 
\[
c \coloneqq \limsup_i \left({\rm coeff}_{D_i}(B_{i,i}) \right),
\]
and observe that ${\rm coeff}_{D_i}(B_{i,i}-B_i) \leq \frac{1}{i}$.
Hence, by~\eqref{ineqcoeff}, for every $\delta>0$ we may find $i$ large enough such that
\[
{\rm coeff}_{D_i}(\Gamma_i) \in (c-\delta,c). 
\]
Thus, passing to a subsequence, we obtain an infinite increasing sequence
\[
{\rm coeff}_{D_i}(\Gamma_i) = {\rm glct}( K_{X_i}+\Gamma'_i + M_i \mid D_i),
\]
contradicting~\cite[Theorem 1.5]{BZ16}.
\item 
Since $X$ is of Fano-type over $Z$, then the minimal model program for any divisor on $X$ over $Z$ terminates~\cite[Corollary 1.3.2]{BCHM}.
Hence, we may run a minimal model program for $-(K_X+B_m+M)$, which terminates with $(X'',B''_m+M'')$. 
Moreover, since $K_X+B+\Omega+M$ is $\qq$-trivial over $Z$, 
we conclude that the above minimal model program is $(K_X+B+\Omega+M)$-trivial,
therefore $(X'',B''+\Omega''+M'')$ is generalized log canonical.
Hence $(X'',B''+M'')$ is generalized log canonical as well.
Observe that all the assumptions on $(X,B+M)$ are preserved when running this minimal model program.
Therefore, up to replacing $(X,B+M)$ with $(X'',B''+M'')$, the same argument as in the first step proves that 
$(X'',B''_m+M'')$ is generalized log canonical for $m$ large enough, since all the assumptions of the proposition are preserved by running a minimal model program.
\item Assume this is not true. Then there exists a sequence of generalized log canonical pairs
$(X_i,B_i+M_i)$ as in the statement, such that the minimal model program for the divisor
$-(K_{X_i}+B_{i,i}+M_i)$ terminates with a Mori fiber space $X_i'' \rightarrow Z_i$ for the 
$\qq$-divisor $-(K_{X''_i}+B''_{i,i}+M''_i$).
Observe that $-(K_{X_i''}+B_i''+M_i'')$ is a pseudo-effective divisor over $Z_{i}$, being $-(K_{X_i}+B_i+M_i)\sim \subs \qq,Z. \Omega$ effective over $Z$.
Hence, we conclude that $K_{X_i''}+B_i''+M_i''$ is
anti-nef over $Z_i$.
On the other hand, the divisor $K_{X_i''}+B_{i,i}''+M_i''$ is ample over $Z_i$.
Hence, perturbing the coefficients of $B_i''$ as in the first step, we can produce boundaries
$B_i''\leq \Gamma_i'' < B_{i,i}$ and prime divisors $D''_i$ which are ample over $Z_i$ such that
\[
{\rm coeff}_{D''_i}(B''_i) \leq {\rm coeff}_{D''_i}(\Gamma_i'') < {\rm coeff}_{D''_i}(B_{i,i}''),
\]
all the remaining coefficients of $\Gamma_i''$ belong to $\overline{\Lambda}$, and 
\[
-(K_{X_i''}+\Gamma_i''+M_i'') \equiv 0/Z_i''.
\]
Passing to a subsequence, we may assume that ${\rm coeff}_{D_i''}(\Gamma_i'')$ forms an infinite increasing sequence, so the coefficients of the divisors $\Gamma_i''$ belong to an infinite set satisfying the descending chain condition. By restricting to a general fiber of $X''_i \rightarrow Z_i$, we get a contradiction by the global ascending chain condition for generalized pairs (see ~\cite[Theorem 1.6]{BZ16}).
\end{enumerate}
Now we turn to prove the last statement. Assume that $K_{X''}+B_m^{+}+M''$ is a strong $(\epsilon,n)$-complement for
$-(K_{X''}+B''+M'').$ Let $Y$ be a log resolution of the minimal model program $X \dashrightarrow X''$ with birational projective morphisms $\psi\colon Y\rightarrow X$ and $\phi\colon Y\rightarrow X''$.
Then, we can write
\[
\psi^*(K_{X}+B_m+M)+E = \phi^*(K_{X''}+B''_m+M''),
\]
where $E$ is an effective divisor. Let 
\[
B^+= B_m + \psi_*(E+\phi^*(B^+_m - B''_m)).
\]
Then, we conclude that $K_{X}+B^{+}+M=\psi_*\phi^*(K_{X''}+B^+_m+M'')$.
Hence, $(X,B^{+}+M)$ is generalized $\epsilon$-log canonical and
$n(K_X+B^{+}+M)$ is Cartier over a neighborhood of $z$.
Since $B^{+}\geq B_m \geq B$, we conclude the $(X,B^{+}+M)$ is a strong
$(\epsilon,n)$-complement of the generalized pair $(X,B+M)$.
\end{proof}

\begin{proof}[Proof of Theorem~\ref{globalen}]
In the case that $\epsilon=0$ we will reduce the statement to~\cite[Theorem 1.10]{Bir16a}. When
$\epsilon>0$, we will use a generalized version of the boundedness of Fano varieties to prove the statement.

First, we reduce to the case when $X$ is a $\qq$-factorial variety.
Since $X$ is of Fano-type, we can find a boundary divisor $\Gamma$ on $X$ such that the pair
$(X,\Gamma)$ is klt~\cite[2.10]{Bir16a}.
Then, we can take a small $\qq$-factorialization $\pi \colon Y \rightarrow X$ of $X$.
Let $B_Y$ denote the strict transform of $B$ on $Y$, and let $M_Y$ be the trace of $M$ on $Y$,
Then, the push-forward to $X$ of an strong $(\epsilon,n)$-complement for $(Y,B_Y+M_Y)$ is a strong $(\epsilon,n)$-complement for $(X,B+M)$ (see, e.g. ~\cite[6.1.(2)]{Bir16a}).
Thus, replacing $(X,B+M)$ with $(Y,B_Y+M_Y)$, we may assume that $X$ is a $\qq$-factorial variety.

Now we turn to prove the statement when $\epsilon=0$.
Assume that $\epsilon=0$, meaning that the generalized pairs $(X,B+M)$ are generalized log canonical.
By Lemma~\ref{incrcoeff}, there exists $m \in \mathbb{N}$ such that for every generalized pair $(X,B+M)$ as in the statement 
the following conditions hold:
\begin{enumerate}
\item $(X,B_m+M)$ is generalized log canonical; 
\item we may run a minimal model program for $-(K_X+B_m+M)$ which terminates with a generalized log canonical pair $(X'',B_m''+M'')$; and 
\item $-(K_{X''}+B_m''+M'')$ is nef.
\end{enumerate}
Observe that the coefficients of $B''_m$ belong to the finite set of rational numbers $\mathcal{C}_m$ defined in Notation \ref{notation}. 
Hence, by~\cite[Theorem 1.10]{Bir16a}, we can find $n$ only depending on $d$, $p$, and $\mathcal{C}_m$ so that 
$K_{X''}+B_{m}''+M''$ has a strong $(0,n)$-complement. 
By Lemma~\ref{incrcoeff}, we conclude that the generalized pair $(X,B+M)$ has a strong $(0,n)$-complement.
This proves the theorem in the case that $\epsilon=0$.

Now, we turn to prove the case where $\epsilon>0$ and $\Lambda$ is finite.
Since $pM'$ is Cartier, then the divisor $pM$ is Weil.
By Lemma \ref{lemma strong complements finite coeff}, every $(X,B+M)$ as in the statement admits a strong $(\epsilon,n)$-complement for some $n$ depending on $d$, $\epsilon$ and $\Lambda$.

Finally, we prove the case in which $M'$ is trivial
and $\Lambda$ is possibly infinite,
i.e., the case of pairs.
Let $\mathcal{P}$ denote the set of varieties $X$ corresponding to the pairs $(X,B)$ as in the statement.
By the proof of Lemma~\ref{lemma strong complements finite coeff}, the varieties in $\mathcal{P}$ belong to a bounded family.
Let $\mathcal{X}\rightarrow T$ be a bounding family.
Thus, there exists a positive number $\mathcal{C}=\mathcal{C}(\mathcal{P})$ so that any $X$ as in the statement admits a very ample line bundle $A$ with $A^d\leq \mathcal{C}$.
We may assume that $A^{d-1}\cdot (-K_X)\leq \mathcal{C}$ for all $X$ in $\mathcal{P}$.
Being ${\rm coeff}(B)>\delta>0$ for some fixed $\delta$ small enough, we get 
\[
\mathcal{C} \geq A^{d-1} \cdot (-K_X)\geq  B \cdot A^{d-1}  \geq 
\delta ({\rm red}(B) \cdot A^{d-1}),
\]
where ${\rm red}(B)$ denotes the reduced structure of $B$.
Hence, we deduce that the log pairs $(X,B)$ belong to a log bounded family.
Thus, up to redefining $\mathcal{X}$ and $T$, we may assume that 
there exists a reduced divisor $\mathcal{B}\subset \mathcal{X}$ bounding the boundaries.

Now, we use the boundedness of $(X,B)$ to prove the statement. Recall that we are assuming that $X$ is $\qq$-factorial.
Arguing by contradiction, we assume that there is a sequence $\lbrace (X_i,B_i) \rbrace \subs i \geq 1.$ satisfying the hypotheses of the statement such that $(X_i,B_i)$ admits no strong $(\epsilon,j)$-complement for $j \leq i$.
Let $\mathcal{B}\sups (i).$ be the divisor supported on $\mathcal{B}$ such that $\mathcal{B}\sups (i). _i =B_i$. Since we have $\coeff (B_i) \in \Lambda$, up to passing to a subsequence, we may assume that $\mathcal{B} \sups (i). \leq \mathcal{B} \sups (i+1).$ and $\Supp (\mathcal{B} \sups (i).) =\Supp (\mathcal{B} \sups (i+1).)$. Since the coefficients lie in $\Lambda$, we can set $\mathcal{B} \sups (\infty).\coloneqq \lim \mathcal{B} \sups (i).$.

Let $t_i \in T$ be the closed point corresponding to $(X_i,B_i)$.
Up to passing to a subsequence and replacing $T$ with the resolution of a subvariety, we may assume that $T$ is a smooth variety and $\lbrace t_i \rbrace \subs i \geq 1.$ is a dense sequence on $T$.
Since $(X_i,\mathcal{B}_i \sups (1).)$ is $\epsilon$-log canonical for all $i$, by Proposition \ref{prop Q-factorial family}, we may assume that $\mathcal{X}$ is $\qq$-factorial.

Let $f \colon  \mathcal{X}' \rar \mathcal{X}$ be a log resolution of the pair $(\mathcal{X},\mathcal{B})$.
Up to shrinking $T$, we may assume that the composition $\pi \circ f\colon  (\mathcal{X}',\mathcal{B}' + \mathcal{E}') \rar T$ is a log smooth family, where $\mathcal{B}' \coloneqq f_* \sups -1. \mathcal{B}$, and $\mathcal{E}'$ denotes the reduced exceptional divisor of $f$. In particular, each fiber $(\mathcal{X}'_i,\mathcal{B}'_i+ \mathcal{E}'_i)$ is a log resolution of $(X_i,B_i)$. Let $\mathcal{P}\sups (i).$ be the divisor supported on $\mathcal{B}'+ \mathcal{E}'$ such that $K_{\mathcal{X}'_i}+\mathcal{P}_i \sups (i).$ is the log pull-back of $K_{X_i}+B_i$.
By construction, $K_{\mathcal{X}'}+\mathcal{P} \sups (i).$ is the log pull-back of $K_{\mathcal{X}}+\mathcal{B} \sups (i).$, and therefore the sequence $\lbrace \mathcal{P} \sups (i).\rbrace \subs i \geq 1.$ is increasing and admits a limit $\mathcal{P} \sups (\infty).$.
By construction, each $(\mathcal{X}'_i,\mathcal{P}_i \sups (i).)$ is sub-$\epsilon$-log canonical.
Thus, $\coeff(\mathcal{P} \sups (i).) \leq 1-\epsilon$. By continuity, we argue that $\coeff(\mathcal{P} \sups (\infty).) \leq 1-\epsilon$. Thus, the construction guarantees that the pairs $(X_i,\mathcal{B}_i \sups (\infty).)$ are $\epsilon$-log canonical.

Now, as $\mathcal{B}_i \sups (\infty). \geq B_i$, a strong $(\epsilon,n)$-complement for $(X,\mathcal{B}_i \sups (\infty).)$ is also a strong $(\epsilon,n)$-complement for $(X,B_i)$. By Lemma \ref{lemma strong complements finite coeff}, each pair $(X,\mathcal{B}_i \sups (\infty).)$ admits a strong $(\epsilon,n)$-complement, where $n$ is independent of $i$. This provides the needed contradiction, and the claim follows.
\end{proof}

The following statement is a generalization of Lemma~\ref{generalized BAB} that allows us to 
put the generalized pairs $(X,B+M)$ as in the statement of Lemma~\ref{generalized BAB} in a generalized log bounded family.

\begin{theorem} \label{log boundedness}
Let $d$ and $p$ be two natural numbers, 
$\epsilon$ a positive real number, and $\Lambda \subset \qq$ be a set satisfying the descending chain condition
with rational accumulation points. Let $\mathcal{P}$ be the set of generalized pairs $(X,B+M)$ such that: 
\begin{itemize}\label{generalizedlogboundedness}
\item $(X,B+M)$ is generalized $\epsilon$-log canonical of dimension $d$;
\item $-(K_X+B+M)$ is nef and big; 
\item $\coeff ( B)\subset \Lambda$; and 
\item $pM'$ is Cartier.
\end{itemize}
Then $\mathcal{P}$ is generalized log bounded.
\end{theorem}

\begin{proof}[Proof of Theorem~\ref{generalizedlogboundedness}]
By Theorem~\ref{globalen}, for any generalized pair $(X,B+M)\in \mathcal{P}$ there exists a bounded $(0,n)$-complement $(X,B^{+}+M)$.
Since $B^+ \geq B \geq 0$, it suffices to show that the generalized pairs $(X,B^++M)$ are generalized log bounded.

By Lemma~\ref{generalized BAB}, the varieties $X$ corresponding to $\mathcal{P}$ are bounded. Let $\mathcal{X}\rightarrow T$ be a bounding family. Thus, there exists a positive number $\mathcal{C}=\mathcal{C}(\mathcal{P})$ such that every $X$ as in the statement admits a very ample divisor $A$ with $A^d\leq \mathcal{C}$.
Furthermore, we may assume that $A^{d-1}\cdot (-K_X)\leq \mathcal{C}$ for all $X$ in $\mathcal{P}$.

Being $B^+$ effective, $M$ the push-forward of a nef divisor, and $A$ ample, we have $A \sups d-1. \cdot B^+ \geq 0$ and $A \sups d-1. \cdot M \geq 0$. Thus, we get the chain of inequalities
$$
\mathcal{C} \geq A \sups d-1. \cdot (-K_X) \geq A \sups d-1. \cdot (-K_X - B^+) \geq A \sups d-1. \cdot (-K_X - B^+ - M) = 0.
$$
We conclude that $A \sups d-1. \cdot B^+ \leq \mathcal{C}$. 
Recall that $nB^+$ is integral, hence $\coeff (B^+) \subset \lbrace \frac{1}{n}, \ldots , \frac{n-1}{n}, 1 \rbrace$. We conclude that $\Supp (B^+)$ is bounded as well. Let $\mathcal{B}$ be the divisor on $\mathcal{X}$ bounding $\Supp (B^+)$. Since the coefficients of $B^+$ belong to a finite set, there are finitely many divisors $\mathcal{D}_1,\ldots,\mathcal{D}_k$ supported on $\mathcal{B}$ such that for any $(X,B^+ + M)$ we have $B^+ = \mathcal{D}_i|_X$ for some $1 \leq i \leq k$. Therefore, $M \sim_\qq (-\K  \mathcal{X}. - \mathcal{D}_i)|_X$. Thus, $M$ is bounded up to $\qq$-linear equivalence as well, and the claim follows.
\end{proof}

\section{Local strong $(\delta,n)$-complements}

In this section, we prove Theorem~\ref{localconj0n} and Theorem~\ref{localen}.
The former is a generalization of ~\cite[Theorem 1.8]{Bir16a} to the setting of generalized pairs, and the latter is a partial result towards the existence of klt complements in the semi-stable setting.
The proof of the following proposition is analogous to the one in~\cite[Proposition 8.1]{Bir16a}.

\begin{proposition}\label{prop1}
The statement of Conjecture~\ref{encomplement} holds for $\epsilon=\delta=0$ and $\Lambda$ finite, if there exist a boundary divisor $\Gamma\geq 0$ and $0<\beta<1$ satisfying the following
\begin{itemize}
    \item $(X,\Gamma+\beta M)$ is $\qq$-factorial generalized plt;
    \item the divisor $-(K_X+\Gamma+\beta M)$ is ample over $Z$; and 
    \item $S=\lfloor \Gamma \rfloor$ is an irreducible component of $\lfloor B \rfloor$ which intersects the fiber over $z$.
\end{itemize}
\end{proposition}

\begin{proof}
We proceed by induction on the dimension of $X$. 
Recall that $\beta M$ is the push-forward of a nef divisor on a higher model, hence its diminished base locus contains no divisors.
In particular, for every $\lambda>0$, we can find $D'\sim_\qq \beta M+\lambda A$ so that
$(X,D)$ is plt, 
where $D=D'+\Gamma$
(similar to Lemma~\ref{generalized BAB}).
Moreover, if we pick $\lambda >0$ small enough, we 
achieve that $-(K_X+D)$ is ample over $Z$, 
and $S=\lfloor D \rfloor$.
Hence, by the first step of the proof of~\cite[Proposition 8.1]{Bir16a}, we conclude that the morphism $S\rightarrow f(S)$ is a contraction and $S$ is of Fano-type over $f(S)$.

We consider a log resolution $\phi\colon X'\rightarrow X$ of $(X,B+M)$ where $M'$ descends,
$S'$ is the birational transform of $S$, and $\psi \colon S'\rightarrow S$ the corresponding morphism.
By generalized adjunction we can write $K_S+B_S+M_S=(K_X+B+M)|_S$, where the coefficients of $B_S$ belong to a set of hyperstandard coefficients $\Omega$ (see, e.g.~\cite[Lemma 3.3]{Bir16a}), and $pM_S'$ is a Cartier divisor.
By Theorem~\ref{globalen} in the case that $f(S)$ is a point or by the inductive hypothesis in the case that $\dim f(S)\geq 1$, we conclude there exists a $n$-complement $K_S+B_S^+ +M_S$ of $K_S+B_S+M_S$ over $z \in f(S)$, where 
$n$ only depends on $\Omega$, $d-1$, and $p$. Replacing $n$ with a bounded multiple, we may assume that $nB$ 
and $nM'$ are integral divisors.
Observe that replacing $\Gamma$ with $\epsilon \Gamma + (1-\epsilon)B$ for some small positive real number $\epsilon$
does not change the assumptions of the proposition.
Hence, we may assume that the coefficients of $\Gamma-B$ are arbitrarily small.

From now on, we will prove that we can lift the $n$-complement $B_S^{+}$ 
to a complement of $K_X+B+M$ over a neighborhood of $z\in Z$.
In order to do so, we will apply Kawamata--Viehweg vanishing on the log resolution $X'$.
We adopt the following notation
\[
N' \coloneqq -(K_{X'}+B'+M')=-\phi^*(K_X+B+M), \text{ and } 
A' \coloneqq -(K_{X'}+\Gamma'+\beta M')=-\phi^*(K_X+\Gamma+\beta M).
\]
We denote by $T'$ the sum of the components of $B'$ with coefficient one, 
and write $\Delta' \coloneqq B'-T'$ 
We will also consider the integral divisor 
\[
L' \coloneqq -nK_{X'}-nT'-\lfloor (n+1)\Delta' \rfloor - nM'.
\]
We claim that there exists a unique integral divisor $P'$ so that
\[
\Sigma' \coloneqq \Gamma'+n\Delta'-\lfloor (n+1)\Delta' \rfloor + P'
\]
is effective, $(X',\Sigma'+\beta M')$ is generalized plt, $\lfloor \Sigma' \rfloor=S'$, and $P'$ is exceptional over $X$.
Indeed, we can define the divisor $P'$ by declaring
\[
{\rm coeff}_D(P') =
\left\{
	\begin{array}{ll}
		0  & \mbox{if } D=S \\
		-{\rm coeff}_D \lfloor \Gamma'-D' +\{ (n+1)\Delta \} \rfloor & \text{otherwise} 
	\end{array}
\right.
\]
for any prime divisor $D$ on $X'$. 
By the definition and the fact that the coefficients of $\Gamma-B$ can be chosen to be arbitrarily small, 
we deduce that the coefficients of $P'$ are contained in $\{0,1\}$.
Finally, it suffices to check that $P'$ is an exceptional divisor over $X$.
If $D$ is a prime divisor whose image on $X$ is a divisor, 
we have that $nB$ is integral, ${\rm coeff}_D(n\Delta')$ is integral,
hence ${\rm coeff}_D P' = - {\rm coeff}_D \lfloor \Gamma \rfloor$ =0,
finishing the claim.

Observe that 
\[ 
L'+P'-S' = K_{X'}+\Sigma'-S'+A'+nN'+\beta M'
\]
is the sum of the klt pair $K_{X'}+\Sigma'-S'$ and the divisor $A'+nN'+\beta M'$, which is nef and big over $Z$.
Up to shrinking $Z$ around $z$, we may assume that $Z$ is affine, so we may apply relative Kawamata--Viehweg vanishing~\cite[Theorem 1-2-5]{KMM85} to deduce that the restriction homomorphism
\begin{equation}\label{surjectivitylp}
H^0(L'+P')\rightarrow H^0((L'+P')|_{S'})
\end{equation}
is surjective.
We denote by $R_{S'}$ the pull-back to $S'$ of $R_S=B_S^+-B_S$ and define 
\[
G_{S'}= nR_{S'}+n\Delta_{S'}-\lfloor (n+1)\Delta_{S'}\rfloor +P_{S'},
\]
where the subscript $S'$ means restriction to $S'$.
We claim that $G_{S'}$ is an effective integral divisor
and $G_{S'}\sim L_{S'}+P_{S'}$, up to shrinking $Z$ around $z$.
Indeed, observe that we have
\[
G_{S'} \sim (n \Delta' - \lfloor (n+1)\Delta' \rfloor +n N' + P')|_{S'}  =(L'+P')|_{S'}=L_{S'}+P_{S'}.
\]
On the other hand, if ${\rm coeff}_{D}(G_{S'})<0$ for some prime divisor $D$,
we have that ${\rm coeff}_D( n\Delta_{S'} - \lfloor (n+1)\Delta_{S'}\rfloor )<0$ as well.
Hence, there is a prime divisor $C$ of $X'$ such that 
${\rm coeff}_C(n \Delta' - \lfloor (n+1) \Delta' \rfloor)<0$ and $D$ is a prime component of $C|_{S'}$.
However, since $n\Delta'$ is integral, we have that ${\rm coeff}_C( n \Delta' - \lfloor (n+1) \Delta' \rfloor)>-1$,
which implies that ${\rm coeff}_D(G_{S'})>-1$.
In particular, we have ${\rm coeff}_D(G_{S'})\geq 0$, being $G_{S'}$ integral.

Therefore, by the surjectivity of~\eqref{surjectivitylp}, there exists $G'\sim L'+P'$ on $X'$
so that $G'|_{S'}=G_{S'}$.
We denote by $G$ the push-forward of $G'$ to $X$. We define $R' \coloneqq \frac{G'}{n}$ and $R$ its push-forward to $X$.
By construction, we have $-n(K_X+B+M)\sim G = nR$, so $B^+ \coloneqq B+R$ is such that $n(K_X+B^{+}+M)$ is Cartier over a neighborhood of $z$.

Finally, we need to check that $(X,B^{+}+M)$ is generalized log canonical over $z$, meaning that $(X,B^+ +M)$ is a strong $(0,n)$-complement of $K_X+B+M$ over $z$. First, we claim that $(X,B^{+}+M)$ is generalized log canonical around $S$.
Observe that $nR'\sim nN'\sim_{\qq,X} 0$, and $\phi_*nR'=nR$, so $\phi^*(R)=R'$.
Moreover, $nR_{S'}=nR'|_{S'}$ implies that $R_S=R|_S$.
We conclude that 
\[
K_S+B^{+}_S+M_S=(K_X+B^{+}+M)|_S,
\] 
so the latter generalized pair is generalized log canonical around $S$ by inversion of adjunction for generalized pairs (see, Theorem \ref{generalized adj and inv of adj} or \cite[Lemma 3.2]{Bir16a}).
If $(X,B^+ +M)$ is not log canonical over $z$, then there is a generalized log canonical center $W$ which intersects the fiber over $z$ and is disjoint from $S$. We define $\Xi= \epsilon \Gamma + (1-\epsilon)B^+$
and $N=((1-\epsilon)+\beta \epsilon)M$ for a small positive real number $\epsilon$. Hence, $(X,\Xi+N)$ is not generalized log canonical at the generic point of $W$ as well, and $-(K_X+\Xi+N)$ is nef and big over a neighborhood of $z$.
This contradicts the connectedness principle for generalized pairs~\cite[Lemma 2.14]{Bir16a}.
Thus, we conclude that $(X,B^++M)$ is generalized log canonical over a neighborhood of $z$.
\end{proof}

\begin{proposition}\label{prop2}
The statement of Conjecture~\ref{encomplement} holds for $\epsilon=\delta=0$ and $\Lambda$ is finite, if there exist $0\leq \tilde{\Delta}\leq \Delta \leq B$ and $0<\beta<1$ so that 
\begin{itemize}
    \item $-(K_X+\Delta+ \beta M)$ and $-(K_X+\tilde{\Delta}+\beta M)$ are nef and big over $Z$;
    \item some component of $\lfloor \Delta \rfloor$ intersects the fiber over $z$; and 
    \item the generalized pair $(X,\tilde{\Delta}+\beta M)$ is generalized klt.
\end{itemize}
\end{proposition}

\begin{proof}
Taking a $\qq$-factorial generalized dlt modification of $(X,B+M)$,
we may assume that $X$ is $\qq$-factorial and $(X,B+M)$ is generalized dlt (see,~\cite[\S 4]{BZ16} or~\cite{Fil18}).
 Write $-(K_X+\Delta+\beta M)\sim_{\qq,Z} A+G$, where $A$ is ample and $G$ is effective. If $\Supp(G)$ does not contain any generalized non-klt center of $(X,\Delta)$, then the pair $(X,\Delta+\delta G+\beta M)$ is generalized dlt for $\delta$ small enough.
 Moreover $-(K_X+\Delta+\delta G+ \beta M)$ is ample. Indeed, we can write
 \[
 -(K_X+\Delta+\delta G + \beta M)\sim_\qq (1-\delta)\left( \frac{\delta}{1-\delta} A + A + G\right),
 \]
and the right-hand side is the sum of an ample divisor and a nef divisor over $Z$.
Let $S$ be a prime component of $\lfloor \Delta + \delta G \rfloor$ which intersects the fiber over $z$. We can find a small positive real number $\epsilon$, so that
\[
\Gamma \coloneqq \Delta + \delta G  -\epsilon\lfloor \Delta\rfloor + \epsilon S
\]
is a boundary satisfying the following conditions hold:
\begin{itemize}
    \item $(X,\Gamma+\beta M)$ is $\qq$-factorial generalized plt, 
    \item $-(K_X+\Gamma+\beta M)$ is ample over $Z$; 
    \item and $S=\lfloor \Gamma \rfloor$ is an irreducible component of $\lfloor B\rfloor$ which intersects the fiber over $z$.
\end{itemize}
Then, by Proposition~\ref{prop1}, we conclude that the statement of the theorem holds for $(X,B+M)$.

From now on, we may assume that $\Supp(G)$ contains some generalized non-klt center of $(X,\Delta+\beta M)$.
Set $\Delta_s \coloneqq s\tilde{\Delta} +(1-s)\Delta$ for any $s\in [0,1]$.
Observe that $(X,\Delta_s+\beta M)$ is generalized klt for any $s\in (0,1]$ and $-(K_X+\Delta_s+\beta M)$ is nef and big over $Z$.
Indeed, it is the sum of two nef and big divisors.
More precisely, we can write
\begin{equation}\label{nefbig}
-(K_X+\Delta_s+\beta M) = -(K_X+\Delta+\beta M)+\Delta-\Delta_s \sim_{\rr,Z}
A+G+\Delta-\Delta_s.
\end{equation}
Define $\Omega_s \coloneqq \Delta_s + t_s(G+\Delta - \Delta_s)$,
where $t_s$ is the generalized log canonical threshold of $(X,\Delta_s+\beta M)$ with respect to $G+\Delta-\Delta_s$ over $z$.
We claim that, for $s$ small enough, the following holds:
\begin{itemize}
    \item every generalized non-klt place of $(X,\Omega_s+\beta M)$ is a generalized non-klt place of $(X,\Delta+\beta M)$; and
    \item the divisor $-(K_X+\Omega_s+\beta M)$ is ample over $Z$.
\end{itemize}
Let $\pi \colon Y \rightarrow X$ be a log resolution of the couple 
$(X,\Delta+G+\beta M)$ where $M'$ descends. Write $F \coloneqq \Delta-\tilde{\Delta}$
and $\pi^*(F)=\sum_i f_i E_i$, where the $E_i$'s are pairwise distinct prime divisors and the $f_i$'s are positive.
We will write $\pi^*(G)=\sum_i g_i E_i$, where the $g_i$'s are positive numbers.
We can write
\[
\pi^*(K_X+\tilde{\Delta}+\beta M) = K_Y+\sum_i e_i E_i +\beta M_Y,
\]
where the real numbers $e_i$ are at most one.
Thus, we can compute
\[
\pi^*(K_X+\Delta_s+t(G+\Delta-\Delta_s))=
K_Y +\sum_i (e_i+(1-s+ts)f_i+tg_i)E_i+M_Y
\]
for any $s\in [0,1)$ and $t>0$.
Hence, we conclude that 
\[
t_s = \min\left\{
t_{i}(s) \mid E_i\text{ is a divisor on Y}
\right\},
\]
where 
\[
t_{i}(s)= \frac{1-e_i-(1-s)f_i}{g_i+sf_i}.
\]
Since the functions $t_i(s)$ are monotone with respect to $s$.
We conclude that if $t_i(s)=t_s$ for $s$ small enough, then $t_i(0)=t_0=0$.
This means that if $E_i$ is a generalized log canonical place for
$(X,\Omega_s+\beta M)$ for all $s> 0$ small enough, then it is a generalized log canonical place
for $(X,\Delta+\beta M)$. Moreover, observe that $t_s$ converges to zero when $s$ converges to zero. 
This proves the first part of the claim.

Observe that 
\begin{align*}
-(K_X+\Omega_s+\beta M) &=
-(K_X+\Delta)-\Delta_s+\Delta-t_s(G+\Delta-\Delta_s)-\beta M \\
& = -(K_X+\Delta+\beta M) +(1-t_s)(\Delta-\Delta_s) -t_sG \\
& \sim_\rr A+G+(1-t_s)(\Delta-\Delta_s)-t_sG \\
& = (1-t_s)\left( \frac{t_s}{1-t_s}A +A+G+\Delta-\Delta_s \right)\\
& \sim_\rr t_sA - (1-t_s)(K_X+\Delta_s+\beta M),
\end{align*}
where we use the linear equivalence~\eqref{nefbig} in the last step.
Hence, $-(K_X+\Omega_s+\beta M)$ is ample for $s$ small enough.
This proves the second part of the claim.
From now on, we will fix $s$ small enough as in the claim and denote such $\Omega_s$ by $\Omega$.

If $\lfloor \Omega \rfloor\neq 0$, we can perturb the coefficients to guarantee that $\lfloor \Omega \rfloor$ is irreducible and $(X,\Omega+\beta M)$ is generalized plt.
Then, we conclude by Proposition~\ref{prop1}. Thus, we may assume $\lfloor \Omega \rfloor =0$.
Consider $(Y,\Omega_Y+\beta M_Y)$ a $\qq$-factorial dlt modification of $(X,\Omega+\beta M)$.
Observe that $K_X+\beta M$ is a generalized minimal model of $K_{Y}+\lfloor \Omega_Y \rfloor+\beta M_Y$ over $X$, since $(X,\beta M)$ is generalized klt and $\lfloor \Omega \rfloor$ is the reduced exceptional divisor of $Y \rar X$.
Let $X''\dashrightarrow X$ be the last step of a $(\K Y. + \lfloor \Omega_Y \rfloor + \beta M_Y)$-MMP over $X$.
Observe that in the last step of this MMP, 
$\lfloor \Omega_Y\rfloor$ has a unique component, 
and the $(Y,\lfloor \Omega_Y \rfloor +\beta M_Y)$ is generalized dlt,
hence it is generalized plt.
This step is a morphism that contracts a prime divisor $S''$ of $X''$ so that $(X'',S''+\beta M)$ is generalized plt and $-(K_{X''}+S''+\beta M'')$ is ample over $X$.
We denote by $K_{X''}+\Omega''+\beta M''$ and $K_{X''}+\Delta''+\beta M''$
the pull-backs of $K_X+\Omega+\beta M$ and $K_X+\Delta+\beta M$ to $X''$.
Hence, we have that $\lfloor \Omega'' \rfloor = \lfloor \Delta'' \rfloor = S''$. 
It suffices to produce a strong $(0,n)$-complement for 
$K_{X''}+B''+\beta M''$ over $z\in Z$.
If we denote 
\[
\Gamma'' \coloneqq (1-t)\Omega'' + tS''
\]
for $t>0$ small enough,
we have that $(X'',\Gamma''+\beta M)$ is a generalized plt pair,
$-(K_{X''}+\Gamma''+\beta M)$ is ample over $Z$, and 
$\lfloor \Gamma''\rfloor = S''$ is a prime divisor which is a component of $\lfloor B''\rfloor$ that intersects the fiber over $z$. Since $\Gamma'' \leq B''$, we can apply Proposition~\ref{prop1} to conclude that $K_{X''}+B''+ M''$ has a strong $(0,n)$-complement.
\end{proof}

\begin{proof}[Proof of Theorem~\ref{localconj0n}]
In the case that $Z$ is a point, then this is already proved in Theorem~\ref{globalen}, so we may assume $Z$ has dimension at least one.

Denote by $f\colon X\rightarrow Z$ the contraction morphism.
Let $N$ be a prime effective Cartier divisor on $Z$ passing through $z$, and let $t$ be the generalized log canonical threshold of $K_X+B+M$ with respect to $f^*N$ over $z$. Let $(X'',\Omega''+M'')$ be a generalized dlt modification of $(X,\Omega+M)$, where $\Omega \coloneqq B+tf^*N$ (see,~\cite[\S 4]{BZ16} or~\cite[\S 3]{Fil18}). By~\cite[2.13(6)]{Bir16a}, we know that $X''$ is of Fano-type over $Z$. 
There is a boundary $0\leq \Delta'' \leq \Omega''$ such that the coefficients of $\Delta''$ are contained in $\Lambda$,
some component of $\lfloor \Delta'' \rfloor$ is vertical over $Z$ intersecting the fiber over $z$, 
and $B\leq \Delta$, where $\Delta$ is the push-forward of $\Delta''$ to $X$.
We run a minimal model program for $-(K_{X''}+\Delta''+M'')$ over $Z$.
Since $-(K_{X''}+\Delta''+M'')$ is pseudo-effective over $Z$, this minimal model program terminates with a nef divisor $-(K_Y+\Delta''_Y+M_Y)$. 
Observe that $-(K_{X''}+\Omega''+M'')$ is semi-ample over $Z$, as $X''$ is of Fano type over $Z$.
Hence, we can find an effective divisor $D$ so that $(X'',\Omega''+D+M'')$ is generalized log canonical over $Z$ and linearly equivalent to zero over $Z$.
Thus, the above minimal model program is $(K_{X''}+\Omega''+D+M'')$-trivial, and the generalized pair $(Y, \Omega_Y+D_Y+M_Y)$ is generalized log canonical.
Thus, we conclude that $(Y,\Delta_Y+M_Y)$ is generalized log canonical.
In particular, no component of $\lfloor \Delta'' \rfloor$ is contracted by this minimal model program.
Indeed, since we are running a minimal model program for $-(K_{X''}+\Delta''+M'')$, if a component of $\lfloor \Delta''\rfloor$ is contracted, the log discrepancy of the corresponding divisorial valuation becomes negative.
Replacing $(X,B+M)$ with $(Y,\Delta_Y+M_Y)$, we may assume that $\lfloor B\rfloor$ is non-trivial.
Moreover, by Lemma~\ref{incrcoeff} we may assume that the coefficients of $B$ belong to a finite set of rational numbers.

We claim that for every $0<\alpha<\beta<1$ the divisor
$-(K_X+\alpha B+ \beta M)$ is a big divisor over $Z$.
Indeed, since $X$ is of Fano-type over $Z$, there exists a big boundary $B_1$ over $Z$ so that $K_X+B_1\sim_{\qq,Z} 0$ and $(X,B_1)$ is klt.
Moreover, since $-(K_X+B+M)$ is semi-ample over $Z$, we may find $B_2$ effective so that 
$K_X+B+B_2+M\sim_{\qq,Z} 0$ and $(X,B+B_2+M)$ is generalized log canonical.
Hence, we have that
\[
\beta(K_X+B+B_2+M) + (1-\beta)(K_X+B_1) \sim_{\qq,Z} 0.
\]
It follows that
\[
-(K_X+\beta B+\beta M) \sim_{\qq,Z} \beta B_2 +(1-\beta)B_1
\]
is big over $Z$. Therefore, 
\[
-(K_X+\alpha B+\beta M) \sim_{\qq,Z} -(K_X+\beta B + \beta M) + (\beta -\alpha)B 
\]
is big over $Z$ as well.
We define the divisor $\Delta \coloneqq B_{\rm ver}+\alpha B_{\rm hor}$,
where $B_{\rm ver}$ and $B_{\rm hor}$ are the vertical and horizontal components of $B$ over $Z$. Observe that $\Delta=\alpha B$ over the generic fiber of $f$, so we have that the divisor 
$-(K_X+\Delta+\beta M)$ is big over $Z$.

Let $X\rightarrow V$ be the contraction defined by $-(K_X+B+M)$.
Run a minimal model program for $-(K_X+\Delta+\beta M)$ over $V$,
which terminates with a model $X''$ where $-(K_{X''}+\Delta''+\beta M'')$ is nef over $V$. 
We claim that the minimal model program does not contract any component of $\lfloor \Delta \rfloor$. Indeed, the above minimal model program is $(K_X+B+M)$-trivial, hence $(K_{X''}+\Delta''+\beta M'')$ is generalized log canonical. 
If some component of $\lfloor \Delta \rfloor$ is contracted by the above minimal model program, then its center on $X''$ is a generalized non-log canonical center of $(X'',\Delta''+\beta M'')$.
This provides the required contradiction.
Since $-(K_X+B+M)$ is the pull-back of a divisor on $V$ that is ample over $Z$, for a small positive real number $\epsilon$, the generalized pair
\begin{equation}\label{newgenpair1}
    (X'', \epsilon \Delta'' + (1-\epsilon)B'' + (1-\epsilon(1-\beta)) M'')
\end{equation}
is generalized log canonical, the divisor $\lfloor \epsilon \Delta '' +(1-\epsilon)B''\rfloor$ has an irreducible component which intersects the fiber over $z$, 
and 
\[
-(K_{X''}+ \epsilon \Delta'' + (1-\epsilon)B'' + (1-\epsilon(1-\beta)) M'')
\]
is nef and big over $Z$.
Observe that for $\epsilon$ small enough, every curve which intersects the pair in~\eqref{newgenpair1}
trivially also intersects $(K_{X''}+B''+ M'')$ trivially.
Indeed, for $\epsilon$ small enough, 
the only curves which may intersect the pair in~\eqref{newgenpair1},
must be contracted by $X\rightarrow V$.
Hence, we may assume that every curve which intersect the pair in~\eqref{newgenpair1} trivially 
is also a curve contracted by $X''\rightarrow V$.
Replacing $(X,\Delta+\beta M)$ with the generalized pair in~\eqref{newgenpair1}
and $\beta$ with $(1-\epsilon(1-\beta))$, we may assume that $-(K_X+\Delta+\beta M)$ is nef and big over $Z$. Moreover, 
we may assume that every $(K_X+\Delta+\beta M)$-trivial curve over $Z$ is also contracted by $X''\rightarrow V$.

Let $X\rightarrow W$ be the morphism defined by $-(K_X+\Delta+\beta M)$.
By the above assumptions, we know that this morphism is $(K_X+B+M)$-trivial.
Define $\tilde{\Delta}=\gamma \Delta$ for $0\ll \gamma <1$, and let $X\dashrightarrow X''$ be a minimal model program for $-(K_X+\tilde{\Delta}+\beta M)$ over $W$.
Replacing $X''$ with the generalized dlt model of $(X'',\Delta''+\beta M)$, we
may assume that $\lfloor \Delta'' \rfloor$ has a component that intersects the fiber over $z$ non-trivially.
Observe that $-(K_{X''}+\tilde{\Delta}''+\beta M'')$ is nef and big over $W$
and it is the pull-back of a divisor on $W$ which is ample over $Z$.
Hence, $-(K_{X''}+\tilde{\Delta}''+\beta M'')$ is pseudoo-effective and nef over $Z$.
Hence, for a small positive real number $\epsilon$ the generalized pair 
\begin{equation}\label{newgenpair2}
(X'',\epsilon \tilde{\Delta}'' +(1-\epsilon)\Delta'' + \beta M'')
\end{equation}
is generalized klt, and 
\[
-(K_{X''} + \epsilon \tilde{\Delta}'' +(1-\epsilon)\Delta'' + \beta M'')
\]
is nef and big over $Z$, being a positive linear combination of a nef and big divisor and a nef and pseudo-effective divisor. 
Replacing $(X,\tilde{\Delta}+\beta M)$ with
the generalized pair in~\eqref{newgenpair2}, we may assume that there exist $0\leq \tilde{\Delta}\leq \Delta \leq B$ and 
$0<\beta<1$ so that the assumptions of Proposition~\ref{prop2} hold, concluding the proof.
\end{proof}

We turn to prove Theorem~\ref{localen}.
In order to do so, we will prove Lemma~\ref{semistable over curve case}, which is the statement of the theorem when $\dim Z=1$. Then, we will show Proposition~\ref{log canonical closure}, which will allow us to take log canonical closures of morphisms.

\begin{lemma} \label{semistable over curve case}
Let $d$ be a positive integer, $\epsilon$ a positive real number and $\Lambda \subset [0,1]$ a finite set of rational numbers. Then, there exist a positive integer $n$ and a positive real number $\delta$ only depending on $d$, $\epsilon$ and $\Lambda$ such that the following holds. Let $\pi \colon X \rar C$ be a contraction of normal quasi-projective varieties, $C$ a curve and $(X,B)$ an $\epsilon$-log canonical pair of dimension $d$ such that
\begin{itemize}
\item $-(\K X. + B)$ is nef over $C$;
\item $X$ is of Fano-type over $C$;
\item $\coeff (B) \subset \Lambda$; and
\item $\pi \colon  (X,B) \rar C$ is a semi-stable family of semi-log canonical pairs.
\end{itemize}
Then, for every point $o \in C$, there exists a strong $(\delta,n)$-complement for $(X,B)$ over $o$.
\end{lemma}
\begin{proof}
Fix a closed point $o \in C$. Up to shrinking $C$ around $o$, we may assume that $(X_c,B_c)$ is $\epsilon$-log canonical for all $c \neq o$. Since $X$ is of Fano-type over $C$, there exists a boundary $\Delta \geq 0$ such that $(X,\Delta)$ is klt and $-(\K X. + \Delta)$ is nef and big over $C$. Fix a rational number $0 < \alpha \ll 1$. Then $(X,B+\alpha \Delta)$ is $\frac{\epsilon}{2}$-log canonical, and $-(\K X. + B + \alpha \Delta)$ is nef and big over $C$. Thus, up to shrinking $C$, by Lemma \ref{generalized BAB}, the varieties $X_c$ with $c \neq o$ are of Fano-type and belong to a bounded family $\mathcal{P}$ depending just on the data in the statement.

Thus, by \cite[Lemma 2.25]{Bir16a}, there is a positive integer $a$ only depending on the data in the statement such that $a(\K X_c. + B_c)$ is Cartier for $c \neq o$. Then, by Nakayama's lemma, the Weil divisor $a(\K X. + B)$ is Cartier over $C \setminus \lbrace o \rbrace$. Since $X$ is of Fano-type over $C$, by the relative effective basepoint-free theorem \cite[Theorem 2.2.4]{Fuj09}, there exists a positive integer $b$, divisible by $a$ and depending only on the data in the statement, such that $-kb(\K X. + B)$ is $\pi$-free over $C \setminus \lbrace o \rbrace$ for any positive integer $k$. In particular, this implies the existence of bounded klt complements of
$(X,B)$ over the generic point of $C$.

Now, we are left with showing the existence of a suitable complement over the fixed closed point $o \in C$. By inversion of adjunction \cite[Theorem 0.1]{Hac14}, the pair $(X,B+X_o)$ is log canonical. Thus, by \cite[Theorem 1.8]{Bir16a}, it admits a strong $(0,n)$-complement over $o$, where $n$ depends just on the data in the statement. In particular, we can regard it as a strong $(0,nb)$-complement. By Proposition \ref{replace complement in linear series}, we can pick the complement $B^+$ such that
$$
bn(B^+-B) \in |-nb(\K X. + B + X_o)|
$$
is a general element. Notice that, since $C$ is affine, $nb(\K X. + B) \sim nb (\K X. + B + X_o)$. In particular, as $-bn(\K X. + B)$ is $\pi$-free over $C \setminus \lbrace o \rbrace$, we may assume that $(X,B^+)$ is klt over $C \setminus \lbrace o \rbrace$.

Now, $B^+$ is a strong $(0,nb)$-complement for $(X,B)$. By construction, $B^+ - B \geq X_o$. As $X_o \sim_C 0$, the boundary $B' \coloneqq B^+ - X_o$ is a strong $(0,nb)$-complement for $(X,B)$. As the only log canonical centers of $(X,B^+)$ are contained in $X_o$, it follows that $(X,B')$ is klt. Thus, $B'$ is a strong $((nb)\sups -1.,nb)$-complement for $(X,B)$ over $o$.
\end{proof}

\begin{proposition} \label{log canonical closure}
Let $f \colon (X_U,\Delta_U) \rar U$ be a contraction of quasi-projective varieties, $z \in U$ be a closed point and $n$ a positive integer. Assume that $(X_U,\Delta_U)$ is log canonical, $\K X_U. + \Delta_U \sim \subs \qq, U. 0$ and $n\Delta_U$ is integral. Denote by $(U,B_U+M_U)$ the generalized pair induced on $U$ by the canonical bundle formula. Then, we may find a normal projective compactification $Z$ of $U$ and a projective log canonical pair $(X,\Delta) \rar Z$ such that $(X,\Delta) \times_Z U$ is a minimal dlt model of $(X_U,\Delta_U)$, $\K X. + \Delta \sim \subs \qq , Z. 0$, $n\Delta$ is integral and the generic point of each non-klt center of $(X,\Delta)$ maps into $U$. In particular, $(X,\Delta)$ and $(X_U,\Delta_U)$ induce the same generalized pair on $U$.
\end{proposition}

\begin{proof}
Let $Z'$ be a normal projective compactification of $U$. By \cite[Corollary 1.2]{HX13}, there exists a log canonical pair $(X,\Delta)$ mapping to $Z'$ such that $(X,\Delta) \times_{Z'} U = (X_U,\Delta_U)$.
Let $(X',\Delta')$ be a birational model of $X$ obtained taking a minimal dlt model of $(X,\Delta)$ and then removing the components of the boundary mapping to $Z' \setminus U$. In particular, $n\Delta'$ is integral and the generic point of every log canonical center of $(X',\Delta')$ maps to $U$.

By \cite[Theorem 1.6]{HX13}, we can run a $(\K X'. + \Delta')$-MMP with scaling over $Z'$, and it terminates with a good minimal model $(X'',\Delta'')$. Let $Z$ denote the canonical model. In particular, we have $\K X''. + \Delta'' \sim \subs \qq , Z. 0$. By construction, this MMP is the identity over $U$. Thus, since $(X'',\Delta'')$ is isomorphic to $(X',\Delta')$ over $U$, $(X'',\Delta'') \times_Z U$ is a minimal dlt model of $(X_U,\Delta_U)$. This is the model claimed.
\end{proof}

\begin{proof}[Proof of Theorem~\ref{localen}]
We will proceed by induction on the dimension of the base $Z$. The base of the induction is given by Lemma \ref{semistable over curve case}. Thus, we may assume that $\dim Z \geq 2$.

First, we reduce to the case when $z$ is a closed point.
Let $H$ be a general very ample divisor on $Z$, and write $X_H \coloneqq \pi^*H$. Since $X_H$ is a general element of a basepoint-free linear series, it follows that $(X_H,B_H)$ is still $\epsilon$-log canonical, $\coeff(B_H) \subset \Lambda$ and $X_H$ is of Fano-type over $H$. Furthermore, $mK_H$ is Cartier \cite[Proposition 4.5.(3)]{Kol13}. Thus, the morphism $\pi_H  \colon  (X_H,B_H) \rar H$ satisfies the same assumptions in the statement.

Fix $z \in Z$, and assume $\dim \O Z,z. \leq n-1$. Then, $\overline{\lbrace z \rbrace} \cap H \neq \emptyset$. Let $z_H$ be the generic point of an irreducible component of $\overline{\lbrace z \rbrace} \cap H$. Then, by induction on the dimension of the base, $(X_H,B_H)$ admits a strong $(\delta,n)$-complement $B_H^+$ over $z_H$, where $\delta > 0$ and $n \in \nn$ depend just on $d$, $\epsilon$, $m$ and $\Lambda$. Without loss of generality, we may assume that $n\cdot \Lambda \subset \nn$. Notice that, as $\Lambda$ is a finite set, by the construction in Proposition \ref{replace complement in linear series}, we may identify $n(B_H^+-B_H)$ with an element of $|-n(\K X_H . + B_H)|$ over $z_H$.

Now, we may shrink $Z$, and assume that it is affine. Therefore, we have $H \sim 0$. This implies that there is a non-canonical isomorphism $\O X. (\K X.)| \subs X_H. \cong \O X_H. (\K X_H.)$. Thus, by twisting the short exact sequence
$$
0 \rar \O X.(-f^*H) \rar \O X. \rar \O X_H. \rar 0,
$$
with $\mathcal{O}_X(-n(K_X+B))$ we obtain the short exact sequence
\begin{equation} \label{ses lift complement}
0 \rar \O X.(-n(\K X. + B)-f^*H) \rar \O X. (-n(\K X. + B)) \rar \O X_H. (-n(\K X_H. + B_H))\rar 0.
\end{equation}
Notice that the sequence in \eqref{ses lift complement} remains exact, since $\O X.(-n(\K X. + B)-f^*H)$ is torsion-free.

By \cite[Theorem 4.4]{CU15}, we have $R^1 f_* \O X. (-n(\K X. + B))=0$. Therefore, the morphism $f_* \O X. (-n(\K X. + B)) \rar f_* \O X_H. (-n(\K X_H. + B_H))$ is surjective. Hence, we may lift local sections of  $\O X_H. (-n(\K X_H. + B_H))$ to local sections of $\O X. (-n(\K X. + B))$. In particular, there exists $B^+$ such that $n(B^+-B)$ restricts to $n(B_H^+-B_H)$. This is equivalent to saying that $(\K X. + B^+ + f^*H)| \subs X_H. = \K X_H. + B_H^+$. Since $(X_H,B_H^+)$ is klt, by inversion of adjunction \cite[Theorem 5.50]{KM98}, $(X,B^+ + f^* H)$ is plt in a neighborhood of $f^*H$. Notice that, since $\overline{\lbrace z \rbrace} \cap H \neq \emptyset$, this holds true over a neighborhood of $z$. Thus, as $f^*H$ is Cartier, we conclude that $B^+$ is a strong $(n \sups -1., n)$-complement for $(X,B)$ over $z$.

Now, we turn to treat the case when $z$ is a closed point arguing by contradiction.
Let $B^+$ be a $(0,n)$-complement for $(X,B)$ over $z$, where $n$ depends only on $\Lambda$ and $d$ \cite[Theorem 1.8]{Bir16a}. Up to replacing $n$ by a bounded multiple depending only on $\epsilon$, by Proposition \ref{replace complement in linear series}, we may assume that $(X,B^+)$ is klt over the generic point of $Z$. Therefore, there are finitely many closed subvarieties $Y_1, \ldots , Y_k \subset Z$ such that $z \in \cap \subs i=1.^k Y_i$ and $(X, B^+)$ is klt over $Z \setminus \cup \subs i = 1. \sups k. Y_i$. Fix $i$. By the first part of the proof, for a general choice of $F \in |-n(\K X. + B)|$, $(X,B+\frac{F}{n})$ is klt over the generic point of $Y_i$. Since we are intersecting finitely many conditions, namely being klt over $Z \setminus \cup \subs i = 1. \sups k. Y_i$ and being klt over $\eta \subs Y_i.$ for all $i$, by Proposition \ref{replace complement in linear series} we may assume that $(X,B^+)$ is klt over $Z \setminus \lbrace z \rbrace$. Thus, there exists an open neighborhood $U \subset Z$ of $z$ such that $(X,B^+)$ is klt over $U \setminus \lbrace z \rbrace$.

Notice that the existence of the complement over $z$ is local in nature. Thus, up to shrinking $Z$ and $X$ accordingly, we may assume that $(X,B^+)$ is log canonical, $\K X. + B^+ \sim \subs \qq , Z. 0$ and $nB^+$ is integral. 
Since $(X,B^+)$ is klt over the generic point of $Z$, the canonical bundle formula induces a generalized polarized pair $(Z,B_Z+M_Z)$ on $Z$ \cite[Theorem 3.6]{Bir16a}.
Then, by Proposition \ref{log canonical closure} and by \cite[Proposition 6.3]{Bir16a}, we have control over the coefficients of $B_Z$ and on the Cartier index of the moduli part. More precisely, $B_Z$ has coefficients in a set of rational numbers satisfying the descending chain condition, just depending on $\Lambda$ and $d$. Similarly, the Cartier index of the moduli part $M \subs Z'.$ descending on a higher model is a bounded function of $\Lambda$ and $d$. Notice that $(Z,B_Z + M_Z)$ is generalized log canonical, generalized klt on $U \setminus \lbrace z \rbrace$ and $\lbrace z \rbrace$ is a generalized log canonical center \cite[Proposition 4.16]{Fil18}.

Now, we regard $(Z,B_Z+M_Z)$ as being relatively of Fano-type over itself. Then, by Theorem \ref{localconj0n} \cite[cf. Corollary 1.9]{Bir16a},  $l(K_Z+B_Z+M_Z)$ is Cartier around $z$, where $l$ is a positive integer depending just on $d$ and $\Lambda$. Up to taking a multiple depending on the data of the problem, we may assume that $m|l$. Thus, we have that $l(B_Z+M_Z)$ is Cartier around $z$.

Consider $t \coloneqq \glct (X,B | f^*(B_Z+ M_Z))$. Since $(X,B)$ is $\epsilon$-log canonical with $\epsilon>0$, we have $t > 0$. We will show that $t=1$. Let $\pi \colon  Z' \rar Z$ be a log resolution of $(Z,B_Z)$ where $M_{Z'}$ descends. Define $X' \coloneqq X \times_Z Z'$, and let $g \colon  X' \rar Z'$ and $\psi \colon  X' \rar X$ be the induced morphisms. By Proposition \ref{base change normal}, $X'$ is a normal variety.

Since $\Supp(B)$ contains no fibers, we have $\psi^*(\K X. + B) = \psi^* \K X. + \psi_* \sups -1. B$. Write $B' \coloneqq \psi_* \sups -1. B$. Recall that $\psi^* \K X/Z. = \K X'/Z'.$. Thus, we have the following chain of equalities
\begin{equation} \label{equation base change equalities}
    \begin{split}
        \psi^*(\K X. + B + f^*(B_Z+M_Z)) & = \psi^*(\K X/Z. + B) + \psi^* f^*(\K Z. + B_Z+M_Z)\\
        &=\K X'/Z'. + B' + g^*(\K Z'. + B \subs Z'. + M \subs Z'.)\\
        &=\K X'. + B' + g^*(B \subs Z'. + M \subs Z'.).
    \end{split}
\end{equation}

Hence, $(X,B+f^*(B_Z+M_Z))$ is generalized log canonical if and only if so is $(X',B'+g^*(B \subs Z'. + M \subs Z'.))$. Since $M \subs Z'.$ descends on $Z'$, it does not contribute to the singularities of $(X',B'+g^*(B \subs Z'. + M \subs Z'.))$. Also, since $(Z,B_Z+M_Z)$ is generalized log canonical, we have $B \subs Z'. \leq \Supp (B \subs Z'.)$. Since $\Supp (B \subs Z'.)$ is simple normal crossing, by repeated inversion of adjunction \cite{Hac14}, $(X',B'+g^* \Supp(B \subs Z'.))$ is log canonical. Thus, we conclude that  $(X,B+f^*(B_Z+M_Z))$ is generalized log canonical. In particular, we have $t \geq 1$. On the other hand, since $(Z,B_Z + M_Z)$ is not generalized klt at $z$, there exists a prime divisor $P' \subset Z'$ mapping to $z$ such that $\mult \subs P'. B \subs Z'. =1$. Therefore, we have $t=1$.

By construction, $(X,B+f^*(B_Z+M_Z))$ is generalized klt over $U \setminus \lbrace z \rbrace$ and $X_z$ is a generalized log canonical center. Now, up to replacing $n$ by a bounded multiple depending only on the data in the statement, by Theorem \ref{localconj0n} $(X,B+f^*(B_Z+M_Z))$ admits a strong $(0,n)$-complement over $z$. Call it $\Gamma$. Then, as $l f^*(B_Z+M_Z)$ is Cartier over a neighborhood of $z$, $\Gamma$ is a strong $(0,n)$-complement for $(X,B)$ over $z$. By Proposition \ref{replace complement in linear series} and the previous discussion, up to shrinking $U$, we may assume that $(X,\Gamma)$ is klt over $U \setminus \lbrace z \rbrace$. 
Note that $(X,\Gamma)$ is log canonical over $U$
and klt over $U\setminus \lbrace z \rbrace$.

Assume that $(X,\Gamma)$ is klt over $U$.
Then, by Remark~\ref{rmk_klt_complement}, as $n(K_X+\Gamma)\sim_Z 0$, it follows that $(X,\Gamma)$ is a
$\left( \frac{1}{n},n\right)$-complement over $z\in Z$.
In particular, we may take $\delta=\frac{1}{2n}$, and
$(X,\Gamma)$ is $\delta$-log canonical.
Thus, to conclude, it suffices to show that $(X,\Gamma)$ is klt.

We claim that $(X,\Gamma)$ is klt over $U$.
Let $\rho  \colon  Z'' \rar Z$ be a generalized dlt model for $(Z,B_Z+M_Z)$ \cite[Corollary 3.4]{Fil18}. Define $X'' \coloneqq X \times_Z Z''$, which is normal by Proposition \ref{base change normal}, and write $h \colon  X'' \rar Z''$ and $\phi  \colon  X'' \rar X$ for the induced morphisms. By construction $\rho$ is an isomorphism over $U \setminus \lbrace z \rbrace$. Let $P_1, \ldots, P_k$ be the $\rho$-exceptional divisors mapping to $z$. Notice that $k \geq 1$, as $(Z,B_Z+M_Z)$ is not generalized dlt at $z$. Define $Q_i \coloneqq h^* P_i$ for $1 \leq i \leq k$. By construction, $\mult \subs P_i. B \subs Z''.=1$

Now, let $(X'',\Gamma'')$ be the log pull-back of $(X,\Gamma)$ to $X''$. Assume by contradiction that $(X,\Gamma)$ is not klt over $z$. Then, the sub-pair $(X'',\Gamma'')$ has a log canonical center $V''$ with $h(V'') \subset P_i$ for some $1 \leq i \leq k$. Then, $(X'',\Gamma'' + \sigma Q_i)$ is not log canonical for any $\sigma > 0$. On the other hand, by the argument in equation \eqref{equation base change equalities}, $\mult \subs Q_i. (B'' + g^* B \subs Z''.) = 1$. Since $\Gamma$ is a complement for $(X,B+f^*(B_Z+M_Z))$ over $z$, it follows that $(X'',\Gamma'' + Q_i)$ is log canonical. This provides the required contradiction.
\end{proof}

\section{Applications}

In this section, we will use the existence of log canonical complements to prove an effective version of the generalized canonical bundle formula~\ref{genbundfor}. Then, we will prove that the existence of klt complements implies M$^{\rm c}$Kernan's conjecture. 

\begin{lemma}\label{ftovercurve}
Let $(X,B+M)$ be a generalized log canonical pair and $f\colon X \rightarrow C$
be a contraction onto a smooth curve. Assume that $X$ is of Fano-type over some non-empty open set $U\subset C$. Moreover, assume that $K_X+B+M \sim_{\qq,C} 0$, and that the generic point of each generalized non-klt center of $(X,B+M)$ is mapped into $U$. Then $X$ is of Fano-type over $C$.
\end{lemma}

\begin{proof}
Since $X$ is of Fano-type over some non-empty open set $U\subset C$, we can find a boundary $\Gamma$ on $X$ so that $\Gamma$ is big over $U$ and 
$K_X+\Gamma \sim_{\qq,U} 0$ (see, e.g.~\cite[2.10]{Bir16a}).
Hence, since $Z$ is a curve, we may find $D\leq 0$ so that 
$K_X+\Gamma \sim_{\qq,C} D$ and $D$ is mapped to $C\setminus U$.
Since the generic point of each generalized non-klt center of $(X,B+M)$ is mapped into $U$, we conclude that for $t$ small enough the generalized pair 
\[
(X, (1-t)B+ t(\Gamma - D)+(1-t)M)
\]
is generalized klt, and 
\[
K_X+(1-t)B+t(\Gamma - D)) +(1-t)M \sim_{\qq,C} 0.
\]
Now, we reduce to the case when $X$ is $\qq$-factorial.
Indeed, as $(X, (1-t)B+ t(\Gamma - D)+(1-t)M)$ is generalized klt, a generalized dlt modification for it coincides with a $\qq$-factorialization $X^\qq$ of $X$ (see,~\cite[\S 4]{BZ16} or~\cite[\S 3]{Fil18}).
Then, by \cite[2.10]{Bir16a}, $X$ is of Fano-type over $C$ if and only if so is $X^\qq$.
Thus, in the rest of the proof we may assume that $X$ is $\qq$-factorial.

Observe that the boundary divisor 
$\Delta = (1-t)B+t(\Gamma-D)$ is big over $C$.
Write $\Delta\sim_{\qq,C} A+E$, where $E$ is effective and $A$ is an effective ample divisor  over $C$.
We can find $\epsilon$ small enough so that $0\leq\Omega \sim_\qq (1-t)M +\epsilon A$ is a boundary such that
$(X, (1-\epsilon)\Delta + \epsilon E + \Omega)$ is a klt pair with big boundary and
\begin{equation} \label{eq_fix}
K_X+(1-\epsilon)\Delta+\epsilon E+\Omega\sim_{\qq,C} 0.
\end{equation}
Indeed, by construction, we have
\[
K_X+(1-\epsilon)\Delta+\epsilon E+\Omega\sim_{\qq,C} K_X+(1-t)B+t(\Gamma - D)) +(1-t)M,
\]
so \eqref{eq_fix} follows.
Since $(X,(1-t)B+t(\Gamma-D)+(1-t)M)$ is generalized klt and $\Delta=(1-t)B+t(\Gamma-D)$, it it follows that $(X,\Delta+(1-t)M)$ is generalized klt.
Then, as $X$ is $\qq$-factorial, $(X,(1-\epsilon)\Delta)$ is a klt pair.
Since $0 < \epsilon \ll 1$, it follows that $(X,(1-\epsilon)\Delta+\epsilon E)$ is klt.
Finally, let $\pi \colon X' \rar X$ be a model where $M'$ descends.
Let $(X',(1-\epsilon)\Delta')$ denote the trace of $(X,(1-\epsilon)\Delta)$ on $X'$.
Notice that $\Delta'$ may not be effective.
Then, $(X',(1-\epsilon)\Delta'+\epsilon \pi^* E)$ is a klt sub-pair.
Since $M'$ is nef and $\pi^*A$ is nef and big over $C$, then $(1-t)M'+\epsilon \pi^*A$ is nef and big over $C$.
Therefore, there exists an effective $\mathbb{Q}$-divisor $F'$ on $X'$ such that
\[
(1-t)M'+\epsilon \pi^*A \sim_{\mathbb{Q},C} A'_k + F'_k
\]
for all positive integers $k$, where $A'_k$ is an ample $\mathbb{Q}$-divisor and $F'_k \coloneqq \frac{F'}{k}$ \cite[Example 2.2.19]{Laz04a}.
Since $(X',(1-\epsilon)\Delta'+\epsilon \pi^* E)$ is sub-klt, if we choose $A'_k$ generically in its $\mathbb{Q}$-linear equivalence class and $k \gg 1$, the sub-pair $(X',(1-\epsilon)\Delta'+\epsilon \pi^* E + A'_k + F'_k)$ is still sub-klt.
Thus, we define $\Omega \coloneqq \pi_*(A'_k + F'_k)$.
Then, by construction, the pair $(X,(1-\epsilon)\Delta + \epsilon E + \Omega)$ is klt.
Therefore, by~\cite[Lemma 2.11]{Bir16a}, we conclude that $X$ is of Fano-type over $C$.
\end{proof}

\begin{lemma}\label{coeffdivadj}
Let $p$ be a natural number and $\Lambda \subset \qq$ a set satisfying the descending chain condition with rational accumulation points.
Then, there exists a set $\Omega \subset \qq$ satisfying the descending chain condition 
with rational accumulation points, only depending on $p$ and $\Lambda$, satisfying the following.
Let $(X,B+M)$ be a generalized dlt pair and
$S$ an irreducible component of $\lfloor B\rfloor$.
If ${\rm coeff}(B)\subset \Lambda$, $pM'$ is Cartier, and $K_S+B_S+M_S=(K_X+B+M)|_S$ is defined by generalized divisorial adjunction, then ${\rm coeff}(B_S)\subset \Omega$.
\end{lemma}

\begin{proof}
First, we prove that ${\rm coeff}(B_S)$ belongs to a set of rational numbers which only depends
on $p$ and $\Lambda$.
By taking hyperplane sections, we may assume that $X$ is a surface (see, e.g.,~\cite[Remark 4.8]{BZ16}).
Since a dlt surface is $\qq$-factorial, we have that $K_X+B$ is $\qq$-Cartier. We define $K_S+\bar{B}_S=(K_X+B)|_S$ by adjunction.
By construction we have $\bar{B}_S \leq B_S$.
Let $P$ be a point on $S$. We want to find a formula for 
${\rm coeff}_P(B_S)$. We may assume that ${\rm coeff}_P(\bar{B}_S)\leq
{\rm coeff}_P(B_S)<1$, so $(X,B)$ is plt at $P$.
Shrinking around $P$, we may assume that $X'$ is smooth and $S'\rightarrow S$ is an isomorphism. 

Write $B=\sum_{i=1}^n \lambda_i B_i$ with $B_i$ pairwise different prime divisors and $\lambda_i \in \Lambda$.
By~\cite[Corollary 3.10]{Sho92}, we can write
\begin{equation}\nonumber
{\rm coeff}_P(\bar{B}_S)= 1-\frac{1}{m}+ \sum_{i=1}^n \frac{\alpha_i \lambda_i}{m},
\end{equation}
where $m$ is any natural number so that $mB_1,\dots,mB_n$ are Cartier divisors
and every $\alpha_i$ is a non-negative integer.
Write $f\colon X'\rightarrow X$ and $f^*(M)=M'+E'$, where $E'$ is an effective divisor.
Since $pM'$ is Cartier, we conclude that $pM$ is Weil.
Let $m'$ be any positive integer divisible by the Cartier index of $pM$. Hence, we can write 
\[
f^*(m'pM)=m'pM'+m'pE',
\]
and $m'pE$ is a Cartier divisor. In particular, $m'pE'$ is integral.
By definition we have that $B_S=\bar{B}_S+E_S$, where $E_S$ is the push-forward of $E'|_{S'}$.
We deduce that we can write
\begin{equation}\label{coeffbs}
{\rm coeff}_P(B_S)= 1-\frac{1}{m}+ \sum_{i=1}^n \frac{\alpha_i \lambda_i}{m}+
\frac{\beta}{m'p},
\end{equation}
where $\beta$ is a non-negative integer.
We conclude that ${\rm coeff}_P(B_S)$ belongs to a set $\Omega$ only depending on $\Lambda$ and $p$. Since $m,m',\alpha_i,\beta$ are naturals
and $\lambda_i \in \Lambda \subset \qq$, we have that $\Omega\subset \qq$.

Now, we turn to prove that $\Omega$ satisfies the descending chain condition
and has rational accumulation points.
Indeed, let us assume that we have a sequence $(X_i,B_i+M_i)$ and $S_i$ as before so that $c_i={\rm coeff}_{P_i}{S_i}$ is an infinite sequence of pair-wise different rational numbers.
Furthermore assume that the sequence $c_i$ has a unique accumulation point $c_\infty$.
If $c_\infty=1$, then the sequence does not violate the descending chain condition and the accumulation point is rational.
We may assume that $c_\infty <1$. In particular, we can write
$c_i<1-\epsilon$ for some $\epsilon$ small enough.
By~\cite[Proposition 3.9]{Sho92}, we know that there exists a constant $l$ only depending on $\epsilon$ so that the Cartier index of any Weil divisor 
on $X_i$ is bounded by $l$, up to passing to a subsequence.
By the equality~\eqref{coeffbs}, we can write
\[
c_i = 1-\frac{1}{l}+\sum_{i=1}^n\frac{\alpha_i \lambda_i}{l} + \frac{\beta}{lp}
\]
where the $\alpha_i$'s and $\beta$ are positive integers.
From the inequality $\alpha_i<1$, we deduce that $n,\alpha_i$ and $\beta$ belong to a finite family of positive integers.
Hence, passing to a subsequence, we may assume that $n$, the $\alpha_i$'s and
$\beta$ are fixed positive integers.
Therefore, since the $\lambda_i$'s satisfy the descending chain condition, we conclude that the $c_i$'s satisfy the descending chain condition as well.
Moreover, we can write
\[
c_\infty= 1-\frac{1}{l}+\sum_{i=1}^n \frac{\alpha_i \bar{\lambda}_i}{l}+\frac{\beta}{lp},
\]
where $\bar{\lambda}_i\in \overline{\Lambda}\subset \qq$.
We conclude that $\overline{\Omega}\subset \qq$.
\end{proof}

\begin{proposition} \label{prop pull-back cartier}
Let $f \colon  X \rar Z$ be a contraction between normal quasi-projective varieties.
Let $D_1$ and $D_2$ be a Cartier divisor and a $\qq$-Cartier divisor on $Z$, respectively.
Assume that $f^*D_1 \sim f^*D_2$.
Then, $D_2$ is a Cartier divisor, and we have $D_1 \sim D_2$.
\end{proposition}

\begin{proof}
We will subdivide the proof into three small claims.

{\it Claim 1:} Let $g  \colon  U \rar V$ be a projective surjective morphism between normal varieties, and $D$ a $\qq$-Cartier $\qq$-divisor on $V$.
If $g^*D=0$, then $D=0$.

Recall that the pull-back of a $\qq$-Cartier $\qq$-divisor $D$ is defined as $k \sups -1. g^*(kD)$, where $k \in \zz \subs \geq 1.$ and $kD$ is Cartier.
Therefore, it suffices to prove the claim in the case $D$ is Cartier, since $g^*(kD)=0$ if and only if $k \sups -1.g^*(kD)=0$.
Assume that $D \neq 0$.
Then, there exists a prime divisor $P \subset V$ such that ${\rm mult}_P(D)=a \neq 0$.
Then, over the generic point $\eta_P$ of $P$, we have $g^*D=ag^*P$.
As one can compute in the case of curves, $g^*P$ is a non-zero effective divisor whose support over $\eta_P$ consists of the prime divisors in $V$ dominating $P$.
As $a \neq 0$, this proves the claim.

\iffalse
{\it Claim 2:} Let $D_1$ and $D_2$ be as in the statement.
Then, we have $D_1 \sim_\qq D_2$.

Fix $k \in \nn$ such that $kD_2$ is Cartier.
Then, we have
\begin{equation*}
\begin{split}
\O Z. (kD_1) \otimes \O Z. (-kD_2) & \simeq f_*(f^* \O Z. (kD_1) \otimes f^* \O Z. (-kD_2)) \\
&\simeq f_*(\O X. (f^*(kD_1)) \otimes \O X. (-f^*( kD_2))) \\
&\simeq f_*(\O X. (f^*(kD_1)-f^*(kD_2))) \\
&\simeq f_*(\O X.)\\
&\simeq \O Z.,
\end{split}    
\end{equation*}
where we repeatedly used the fact that $f_*\O X. \simeq \O Z.$ and the projection formula.
\fi

{\it Claim 2:} The statement holds if $D_1$ and $D_2$ are effective.

By the projection formula and the fact that $f_* \O X. = \O Z.$, we have a natural isomorphism of groups $\mathrm{Hom}_{\O X.}(\O X.,\O X.(f^*D_1)) \simeq \mathrm{Hom}_{\O Z.}(\O Z.,\O Z.(D_1))$ \cite[\S II.5]{Har77}.
In particular, we have that, for every $0 \leq R \sim f^*D_1$, there exists $0 \leq S \sim D_1$ such that $R=f^*S$.
Let $0 \leq \Gamma$ be the Cartier divisor on $Z$ such that $f^*\Gamma = f^*D_2$.
Since $f^*(\Gamma - D_2)=0$, it follows from Claim 1 that $\Gamma = D_2$.
In particular, $D_2$ is Cartier and $D_2 \sim D_1$.

{\it Claim 3:} Let $D_1$ and $D_2$ be as in the statement.
Then, we may assume that they are effective.

Let $E$ be a reduced divisor such that $\Supp(D_1) \subset E$ and $\Supp(D_2) \subset E$.
Then, as $Z$ is quasi-projective, we may find an effective Cartier divisor $F$ such that $E \leq F$.
Then, for $k \ll 0$, we have $D_1 + k F \geq 0$ and $D_2 + k F \geq 0$.
By Claim 2, we have $D_1 + k F \sim D_2 + k F$.
Then, we have $D_1 \sim D_2$, and the statement follows.
\iffalse
Now, since $Z$ is a quasi-projective curve, we may embed it into a smooth projective curve $\overline{Z}$.
Then, we can find a projective closure $\overline{X}$ of $X$ such that $\overline{X}$ is normal, $\overline{f}  \colon  \overline{X} \drar \overline{Z}$ is a morphims, and $X= \overline{X}\times_{\overline{Z}} Z$.
Notice that, by the Stein factorization, $\overline{f}$ has connected fibers.
Notice that this procedure does not affect the assumptions on $D_1$, $D_2$ and the respective pull-backs.
Therefore, up to replacing $f  \colon  X \rar Z$ with $\overline{f}  \colon  \overline{X} \rar \overline{Z}$, we may assume that $X$ and $Z$ are projective.

Now, let $H$ be a Cartier divisor on $Z$ such that $D_1 + H \geq 0$ and $D_2 + H \geq 0$.
By the projection formula, we have a natural identification between $H^0(X,\O X.(f^*(D_1+H)))$ and $H^0(Z,\O Z. (D_1+H))$. In particular, since these are finitely dimensional vector spaces, every Cartier divisor in $|H^0(X,\O X.(f^*(D_1+H)))|$ is the pull-back of a Cartier divisor in $|H^0(Z,\O Z. (D_1+H))|$. In particular, since $f^*(D_2+H) \in |H^0(X,\O X.(f^*(D_1+H)))|$, we have $D_2+H \in |H^0(Z,\O Z. (D_1+H))|$. As $H$ is Cartier, we conclude that $D_2$ is Cartier, and that $D_1 \sim D_2 $.
\fi
\end{proof}

The proof of Lemma~\ref{fanotypeoveropen} and Theorem~\ref{genbundfor} are similar to the proof of~\cite[Proposition 6.3]{Bir16a}.

\begin{lemma}\label{fanotypeoveropen}
Let $d$ and $p$ be two natural numbers and $\Lambda \subset \qq$ be a set satisfying the descending chain condition with rational accumulation points. 
Then, there exist a natural number $q$ and a set $\Omega\subset \qq$ satisfying the descending chain condition with rational accumulation points, only depending on $d,p$ and $\Lambda$, satisfying the following.
Let $f\colon X \rightarrow Z$ be a contraction between positive dimensional normal quasi-projective varieties, $(X,B+M)$ be a generalized log cannical pair of dimension $d$ such that
\begin{itemize}
    \item $K_X+B+M\sim_{\qq,Z} 0$;
    \item $X$ is of Fano-type over some non-empty open set of $Z$;
    \item ${\rm coeff}(B)\subset \Lambda$; and 
    \item $pM'$ is Cartier.
\end{itemize}
Then the generalized pair $(Z,B_Z+M_Z)$ obtained by the generalized canonical bundle formula satisfies
\begin{itemize}
    \item ${\rm coeff}(B_Z)\subset \Omega$; and
    \item $qM_Z$ is a Weil divisor.
\end{itemize}
\end{lemma}

\begin{proof}
We will prove the existence of $\Omega$ and $q$ by induction on the dimension of $Z$.
We will use the existence of log canonical complements and the generalized canonical bundle formula to produce $q$.

Fix a point $z \in Z$ over which $X$ is of Fano-type. Then, by Theorem~\ref{localconj0n}, we can find a strong $(0,q)$-complement $K_X+B^++M$ of $K_X+B+M$ over $z$. Since $\K X. + B + M \sim_{\qq,Z} 0$, we have that $B^+=B$ over the generic point of $Z$. Hence, we may find $L$ on $X$ and $L_Z$ on $Z$ so that
$q(K_X+B+M)\sim qL$ and $qL=qf^*L_Z$.
Therefore, by the generalized canonical bundle formula~\cite[Theorem 1.4]{Fil18} with respect to $L_Z$, we may write
\[
q(K_X+B+M) \sim qL = qf^*L_Z=qf^*(K_Z+B_Z+M_Z).
\]
Now, we turn to prove the existence of $\Omega$.
Let $H$ be a general hyperplane section on $Z$ and $G \coloneqq f^*H$. Denote by $g\colon G\rightarrow H$ the induced morphism. Write $(K_G+B_G+M_G)=(K_X+B+G+M)|_G$, 
and let $K_H+B_H+M_H$ be the generalized pair obtained by the generalized canonical bundle formula for the morphism $g$ and $K_G+B_G+M_G$.
Observe that $K_G+B_G+M_G\sim_{\qq,H} 0$ and $G$ is of Fano-type over an open set of $H$.
Let $D$ be a prime divisor on $Z$, and let $C$ be a component of $D\cap H$.
Let $t$ be the generalized log canonical threshold of $f^* D$ with respect to $(X,B+M)$ over the generic point of $D$.
By generalized divisorial inversion of adjunction, we have that $t$ is the log canonical threshold of $g^* C$ with respect to $(G,B_G+M_G)$ over the generic point of $C$.
Therefore we conclude that 
${\rm coeff}(B_Z)= {\rm coeff}(B_H)$.
By Lemma~\ref{coeffdivadj}, we know that $pM'_G$  is Cartier and ${\rm coeff}(B_G)$ belongs to a set only depending on $d$, $p$, and $\Lambda$ and satisfying the descending chain condition with rational accumulation points. Therefore, by repeated hyperplane cuts, we may reduce to the case when $Z$ is a smooth curve.

From now on, we assume that $\dim Z=1$. Thus, by Lemma~\ref{ftovercurve}, we have that $X$ is of Fano-type over $Z$. Let $z\in Z$ be a closed point and $t$ the generalized log canonical threshold of $(X,B+M)$ with respect to $f^*z$.
Denote by $(X'',\Gamma''+M'')$ a generalized dlt model of $(X,B+tf^*z+M)$.
There exists a boundary divisor $B''$ such that $B''\leq \Gamma''$, $\coeff(B'') \subset \Lambda$, $\lfloor B''\rfloor$ has a component mapping to $z$ and $\tilde{B}\leq B''$, where $\tilde{B}$ denotes the strict transform of $B$ on $X''$.
Since
$X''$ is of Fano-type over $Z$, we may run a minimal model program for $-(K_{X''}+B''+M'')$ over $Z$,
and denote by $Y$ the resulting model.
Denote by $M_Y$ and $B_Y$ the push-forward of $M''$ and $B''$, respectively.
Since $B''\leq \Gamma''$, $-(K_Y+B_Y+M_Y)$ is nef over $Z$.
Furthermore, since $B'' \leq \Gamma''$ and the MMP is trivial for $\K X''. + \Gamma'' + M''$, $(Y,B_Y+M_Y)$ is generalized log canonical.
By Theorem~\ref{localconj0n}, there exists a strong $(0,q)$-complement $B^+_Y$ of $K_Y+B_Y+M_Y$ over $z\in Z$.
Therefore, there exists a strong $(0,q)$-complement ${B''}^{+}$ of
$K_{X''}+B''+M''$ over $z\in Z$ (see~\cite[6.1.(3)]{Bir16a}).
Since both $K_{X''}+{B''}^{+}+M''$
and $K_{X''}+{B''}+M''$ are relatively trivial over the base, we conclude that ${B''}^+-B''=af^*z$
for some real number $a$.
Moreover, since $K_{X''}+{B''}^{+}+M''$ has a log canonical center mapping onto $z$,
we deduce that $a=t$.
Let $S$ be a component of $f^*Z$ and define $b\coloneqq \mult_S B''$, $b^+ \coloneqq \mult \subs S. {B''}^+ $, and
$m \coloneqq \mult_S f^*z$.
Then, we have that 
\[
\mult_z (B_Z) = 1+\frac{b-b^+}{m}
\]
(with $b\leq b^+\leq 1$) belongs to the set
\[
\Omega \coloneqq \left\{ 
1+\frac{b-b^+}{m} \mid 
b\in \Lambda, b^+\in \nn\left[ \frac{1}{q} \right],
m\in \nn
\right\},
\]
which satisfies the descending chain condition and has rational accumulation points.
Observe that $\Omega$ only depends on 
$q$ and $\Lambda$. In particular, it only depends on
$d,p$ and $\Lambda$.\

Finally, we need to prove that $qM_Z$ is a Weil divisor.
First, we reduce to the case when $Z$ is a curve.
Let $H$ be a general hyperplane section on $Z$, $G$ it's pull-back to $X$, $g\colon G\rightarrow H$ the induced morphism, $H'\sim H$ a general member of the linear system $|H|$, $D$ a prime divisor of $Z$ and $C$ a prime component of $D\cap H$. We can write
$K_H \coloneqq (K_Z+H')|_H$, being $H$ a general hyperplane section.
We write 
\[
M_H  \coloneqq (L_Z+H')|_H - (K_H+B_H),
\]
so we have 
\[
q(K_G+B_G+M_G) \sim 
q(L+G)|_G \sim
qg^*(L_Z+H'|_H) \sim
qg^*(K_H+B_H+M_H).
\]
Hence, $M_H$ is the moduli part induced by 
$(G,B_G+M_G)$ over $H$.
On the other hand, we have that
\[
(B_H+M_H)=(B_Z+M_Z)|_H.
\]
We know that 
${\rm coeff}_{C}(B_H)={\rm coeff}_{D}(B_Z)$,
therefore we have that 
${\rm coeff}(M_Z)= {\rm coeff}(M_H)$. So it suffices to prove that $qM_H$ is Weil.

From now on, we may assume that $\dim Z=1$. Thus, by Lemma~\ref{ftovercurve}, we may assume that $X$ is of Fano-type over $Z$.
Let $V\subset Z$ be an open subset so that ${\rm Supp}(B_Z)\subset Z\setminus V$.
Thus, we can write
\[
\Theta = B+ \sum_{z\in Z\setminus V} t_z f^* z,
\]
where $t_z$ is the generalized log canonical threshold of 
$(X,B+M)$ with respect to $f^*z$.
Let $\Theta_Z$ be the boundary part of generalized adjunction with respect to $(X,\Theta +M)$,
then we have that 
\[
\Theta_Z = B_Z +\sum_{z\in Z\setminus V} t_z z.
\]
By definition of $B_Z$, the divisor $\Theta_Z$ is reduced.
Now, fix $z \in V$.
Then, $(X,\Theta + f^*z + M)$ is generalized log canonical.
By Theorem \ref{localconj0n}, it admits a strong $(0,q)$-complement over $z$.
Since $(X,\Theta + f^*z + M)$ is not generalized klt over $z$, it follows that $q(\K X. + \Theta + f^*z + M)\sim 0$ over some neighborhood of $z$.
Then, as $f^* z$ is Cartier, we conclude that $q(\K X. + \Theta + M)\sim 0$ over some neighborhood of $z$.
Similarly, we can argue that $q(\K X. + \Theta + + M)\sim 0$ over some neighborhood of $z$, for $z \in Z \setminus V$.
Thus, we have that $K_X+\Theta+M$ is a strong $(0,q)$-complement of $K_X+B+M$ over $z\in Z$, for every point $z$. In particular, we have that
$q(K_X+\Theta +M)\sim_{Z} 0$.
Therefore, we have that
\begin{equation*}
    \begin{split}
    q(K_X+\Theta +M) &=
q(K_X+B+M) + q(\Theta -B)\\
&\sim
qf^*(K_Z+B_Z+M_Z)+ qf^*(\Theta_Z-B_Z)\\
& =qf^*(K_Z+\Theta_Z+M_Z).
    \end{split}
\end{equation*}
In particular, we have that $qf^*(K_Z+\Theta_Z+M_Z)$ is a Cartier divisor. Also, $q(\K X. + \Theta + M)$ is linearly equivalent to the pull-back of a Cartier divisor on $Z$. Therefore, by Proposition \ref{prop pull-back cartier}, we conclude that $q(K_Z+\Theta_Z+M_Z)$ is a Cartier divisor. Since $q(K_Z+\Theta_Z)$ is integral, we conclude that $qM_Z$ is integral as well.
\end{proof}

\begin{proof}[Proof of Theorem~\ref{genbundfor}]
Due to Lemma~\ref{fanotypeoveropen}, it suffices to prove that $qM_{Z}'$ is Cartier.
As in the above proof, we can find $L$ on $X$ and $L_Z$ on $Z$ so that
$q(K_X+B+M)\sim qL = qf^*(L_Z)$.

Let $Z_0 \rightarrow Z$ be a high resolution of singularities of $Z$ and 
$X_1$ a log resolution of the generalized pair $(X,B+M)$ so that 
the rational map $X_1\dashrightarrow Z_0$ is a morphism and $M'$ descends on $X_1$.
By the assumptions, we may find a non-empty open set $U\subset Z$ such that $Z_0\rightarrow Z$ is an
isomorphism over $U$, $X$ is of Fano-type over $U$, and $(X,B+M)$ is generalized klt over $U$.
Denote by $U_0$ its inverse image on $Z_0$.
Moreover, we denote by $X_0$ the normalization of the main component of the fiber product $Z_0 \times_Z X$.
Let $\Delta_1$ be the sum of the birational transform $B_1$ of $B$ and the reduced exceptional divisor of $X_1 \rightarrow X$ with all the components mapping outside $U_0$ removed.
We denote by $M_1$ the trace of the birational divisor inducing $M$ on $X_1$.
Observe that $(X_1,\Delta_1+M_1)$ is a generalized klt pair over $U_0$.
By the negativity lemma~\cite[Lemma 3.39]{KM98} and the fact that $(X,B+M)$ is generalized klt over $U$, the relative diminished base locus over $X_0$ of $K_{X_1}+\Delta_1+M_1$ contains all exceptional divisors over $X_0$ with center in $U_0$.
We run a minimal model program for $K_{X_1}+\Delta_1+M_1$ over $X_0$ with scaling of an ample divisor $A_1$. After finitely many steps, 
all these exceptional divisors with center in $U_0$ are contracted.
Indeed, these divisors are in the relative diminished base locus over $X_0$ of $\K X_1. + \Delta_1 + M_1 + \lambda A_1$ for $0 < \lambda \ll 1$, and $\K X_1. + \Delta_1 + M_1 + \lambda A_1 \sim \K X_1. + \Phi_1$, where $\Phi_1 \geq 0$ and $(X_1,\Phi_1)$ is dlt over $U_0$.
Hence, after finitely many steps the variety is of Fano-type over $U_0$ \cite[cf. 2.13.(6)]{Bir16b}, so the minimal model program terminates over $U_0$ (see, e.g.~\cite[2.10]{Bir16a}).
We call this variety $X_2$.
Denote the strict transform of $B_1$ on $X_2$ by $B_2$, the strict transform of $\Delta_1$ on $X_2$ by $\Delta_2$, and by $M_2$ the trace of $M$ on $X_2$.
Observe that the generalized pair $(X_2,\Delta_2+M_2)$ is a small $\qq$-factorialization of $(X,B+M)$ over $U_0$ (see \cite[Lemma 4.5]{BZ16}).
In particular, we have that 
\[
K_{X_2}+\Delta_2+M_2 \sim_{\qq , U_0} 0,
\]
$\Delta_2+M_2$ is big over $Z_0$, 
and $(X_2,\Delta_2+M_2)$ is generalized klt.
Indeed, all log canonical centers of $(X_1,\Delta_1+M_1)$ are contained 
in $\lfloor \Delta_1\rfloor$, and these divisors are contracted by the minimal model program $X_1\dashrightarrow X_2$.
Since $\Delta_2+M_2$ is big over $Z_0$ and $(X_2,\Delta_2+M_2)$ is generalized klt,
we can write
\[
K_{X_2}+\Delta_2+M \sim_{\qq , Z_0} K_{X_2}+\Delta'_2,
\]
where $(X_2,\Delta'_2)$ is a klt pair.
By~\cite[Theorem 1.4]{Bir12} and~\cite[Theorem 1.1]{HMX14},
we can run a minimal model program over $Z_0$ for $K_{X_2}+\Delta'_2$
which terminates with a good minimal model $K_Y+\Delta'_Y$ over $Z_0$.
Hence, we deduce that $K_Y+\Delta_Y+M_Y$ is semi-ample over $Z_0$, 
where $\Delta_Y$ is the push-forward of $\Delta_2$ on $Y$ and $M_Y$
is the trace of $M$ on $Y$. 
Let $f_{Z_1} \colon Y\rightarrow Z_1$ the morphism defined by $K_Y+\Delta_Y+M_Y$ over $Z_0$.
Since $K_Y+\Delta_Y+M_Y$ is $\qq$-linearly trivial over an open set of $Z_0$, we conclude that $Z_1\rightarrow Z_0$ is a birational morphism.
Observe that we have 
\[
K_Y+\Delta_Y+M_Y \sim_{\qq , Z_1} 0.
\]
By Lemma~\ref{fanotypeoveropen}, we conclude that the moduli part $M_{Z_1}$ induced by $(Y,B_Y+M_Y)$ on $Z_1$ is such that
$qM_{Z_1}$ is Weil. 
Moreover, the generalized pairs $(Y,\Delta_Y+M_Y)$ and $(X,B+M)$ have the same moduli part,
and coincide over $U_0$ on a common log resolution $W$ of $X$ and $Y$.
Denote by $K_Y+B_Y+M_Y$ (resp. $L_Y$) the push-forward to $Y$ of the pull-back to $W$ of $K_X+B+M$ (resp. $L$).
The divisor $P_Y \coloneqq \Delta_Y-B_Y$ is vertical over $Z_1$ and $\qq$-trivial over $Z_1$,
hence it is the pull-back of a divisor $P_{Z_1}$ on $Z_1$.
Denote by $(Z_1,\Delta_{Z_1}+M_{Z_1})$ the generalized pair obtained by the generalized canonical bundle formula
for $K_Y+\Delta_Y+M_Y$ over $Z_1$.
Then, we have $\Delta_{Z_1}=P_{Z_1}+B_{Z_1}$,
where $B_{Z_1}$ is the boundary part obtained by the generalized canonical bundle formula 
for $K_Y+B_Y+M_Y$ over $Z_1$.
Observe that we have
\begin{equation*}
\begin{split}
q(K_Y+\Delta_Y+M_Y) &= q(K_Y+B_Y+P_Y+M_Y) \\
&\sim q(L_Y+P_Y) \\
&=qf_{Z_1}^*(L_{Z_1}+P_{Z_1}) \\
&=qf_{Z_1}^*(K_{Z_1}+\Delta_{Z_1}+M_{Z_1}),
\end{split}    
\end{equation*}
where $L_{Z_1}$ is the pull-back of $L_Z$ on $Z_1$.
Here, $M_{Z_1}=L_{Z_1}-(K_{Z_1}+B_{Z_1})$ is the moduli part induced on $Z_1$ by both of 
$(Y,\Delta_Y+M_Y)$ and $(Y,B_Y+M_Y)$.
Therefore, we conclude that $qM_{Z_0}$ is Weil, and hence $qM_{Z_0}$ is Cartier, being $Z_0$ smooth. 
As we may assume that $Z_0$ is the model where $M_{Z_0}$ descends, it follows that $qM_Z'$ is Cartier.
\end{proof}

\begin{proof}[Proof of Corollary~\ref{genadj}]
Let $(X,B+M)$ be a generalized pair and $W\subset X$ be a generalized log canonical center. Let $(Y,B_Y+M_Y)$ be a generalized dlt model of $(X,B+M)$ and $E$ a generalized log canonical place corresponding to $W$. Since $W$ is an exceptional generalized log canonical center, we have that
$E$ is the only generalized log canonical place mapping onto $W$.
Observe that $E$ is normal.
By Lemma~\ref{coeffdivadj}, there exists a set $\Omega_0 \subset \qq$ satisfying the descending chain condition with rational accumulation points
and a natural number $q_0$ depending only on $d,p$ and $\Lambda$,
so we can write 
\[
(K_Y+B_Y+M_Y)|_{E} \sim_\qq K_{E}+B_{E}+M_{E},
\]
${\rm coeff}(B_{E}) \subset \Omega_0$ 
and $q_0 M_{E}'$ is Cartier. 
By the assumption on the exceptionality of $W$, we get that the 
generalized pair $(E,B_E+M_E)$ is generalized klt over the generic point of $W$.
Now, we can apply Theorem~\ref{genbundfor} to the 
generalized pair $(E, B_{E}+M_{E})$
with respect to the morphism
$E \rightarrow W$, and conclude that there exists a generalized pair
$(W, B_{W}+M_{W})$ on $W$, so that 
\[
(K_Y+B_Y+M_Y)|_{W} \sim_\qq K_{W}+B_{W}+M_{W},
\]
${\rm coeff}(B_{W}) \subset \Omega$, and $qM_{W}'$ is Cartier,
where $\Omega \subset \qq$ is a set with the descending chain condition with rational accumulation points and $q$ is a natural number, 
both depending only on $d-1$, $q_0$ and $\Omega_0$, hence
only depending on $d$, $p$ and $\Lambda$.
\end{proof}

For the reader's convenience, we will split Theorem \ref{McKernanthm} into two statements.

\begin{theorem} \label{McKernan part 1}
Conjecture \ref{encomplement} implies Conjecture \ref{McKernanconj}.
\end{theorem}

\begin{proof}
The statement of Conjecture \ref{McKernanconj} is local on the base. Fix a point $z \in Z$. Assuming Conjecture \ref{encomplement}, there exists a strong $(\zeta,n)$-complement $B^+$ for $X$ over $z$, where $\zeta>0$. Up to shrinking $Z$ and $X$ over it accordingly, we may assume that $(X,B^+)$ is $\zeta$-log canonical. Therefore, the general fiber $(X_t,B_t)$ of $f$ is a $\zeta$-log canonical pair. Furthermore, by assumption, $-\K X_t.$ is ample. In particular, by \cite[Theorem 1.1]{Bir16b}, $X_t$ belongs to a bounded family. Since $\coeff(B^+) \subset \lbrace \frac{1}{n}, \ldots \frac{n-1}{n} \rbrace$, by the same argument as in the proof of Theorem \ref{globalen}, the pairs $(X_t,B^+_t)$ belong to a log bounded family depending just on $d$ and $\epsilon$.

Now, by \cite[Theorem 1.3]{Bir16}, there exists boundary $\Delta$ on $Z$ such that $(Z,\Delta)$ is $\delta$-log canonical, where $\delta>0$ depends on $\epsilon$ and $d$. Since $Z$ is $\qq$-factorial \cite[Corollary 3.18]{KM98}, it follows that $Z$ is $\delta$-log canonical.
\end{proof}

Now we will address the second part of the statement of Theorem \ref{McKernanthm}.

\begin{theorem} \label{implication of conjectures}
Let $d$ and $m$ be positive integers and $\epsilon$ a positive real number. Then, there exists a positive real number $\delta$ such that the following holds. If $f  \colon  X \rar Z$ is a Mori fiber space, $X$ is projective, $\epsilon$-log canonical and $\qq$-factorial and $m\K X.$ is Cartier, then $Z$ is $\delta$-log canonical.
\end{theorem}

\begin{proof}
Since $m\K X.$ is Cartier, by \cite[Theorem 2.2.4]{Fuj09}, there is a positive integer $a$ depending just on $m$ and $d$ such that $-a\K X.$ is $f$-free. Up to taking a bounded multiple depending on $\epsilon$, we may assume that $a \geq (1-\epsilon)\sups -1.$.

Since the statement is local in nature, we may fix $z \in Z$ and assume that $Z$ is affine. Then, $\O X.(-a\K X.)$ is basepoint-free. For a general choice of $0 \leq \Gamma \sim -a\K X.$, define $B^+ \coloneqq \frac{\Gamma}{a}$. 
Since $X$ is $\epsilon$-log canonical, $a \geq (1-\epsilon)\sups -1.$, and $\Gamma$ is a general element of a basepoint-free linear series, it follows that $(X,B^+)$ is $\epsilon$-log canonical.
Thus, by construction, $B^+$ is a strong $(\epsilon,a)$-complement for $(X,0)$ over $z$. Then, by the same argument in the proof of Theorem \ref{McKernan part 1}, it follows that $Z$ is $\delta$-log canonical, where $\delta$ depends on $d$, $\epsilon$ and $m$.
\end{proof}

\end{document}